\documentclass{article}
\usepackage{amsmath,amsthm,amsfonts,amscd,amssymb} 
\usepackage[hmargin=2cm, vmargin=3cm]{geometry}
\usepackage{graphicx}
\usepackage{mathrsfs}
\usepackage{titlesec}
\usepackage{authblk}
\usepackage{multirow}
\usepackage{bm}
\usepackage{booktabs}
\usepackage{graphicx}
\usepackage{indentfirst}
\usepackage{color}
\usepackage{algpseudocode}
\usepackage{enumerate}

\usepackage{tikz}
\usepackage{epstopdf}
\usepackage{paralist}
\usepackage{epigraph}
\usepackage{subcaption}
\usepackage{xspace}
\usepackage[group-separator={,}]{siunitx} 
\usetikzlibrary{arrows}
\usepackage[super]{nth}

\renewcommand{\leq}{\leqslant}
\renewcommand{\geq}{\geqslant}

\titlelabel{\thetitle.\quad}

\newtheorem{lemma}{Lemma}
\newtheorem{definition}{Definition}
\newtheorem{proposition}{Proposition}

\newtheorem{conjecture}{Conjecture}

\algtext*{EndIf}
\algtext*{EndWhile}
\algtext*{EndFor}

\DeclareMathAlphabet{\mathscr}{U}{dutchcal}{m}{n}
\SetMathAlphabet{\mathscr}{bold}{U}{dutchcal}{b}{n}
\DeclareMathAlphabet{\mathbscr} {U}{dutchcal}{b}{n}

\graphicspath{{./Figures1/}}
\begin{document}

\title{\bf{On The Poisson Follower Model}}

\author[1]{Natasa Dragovic \thanks{natasa.dragovic@stthomas.edu}}
\author[2]{Francois Baccelli \thanks{francois.baccelli@ens.fr}}
\affil[1]{University of Saint Thomas}
\affil[2]{INRIA Paris}


\date{June 2026}
\maketitle

\begin{abstract}
We introduce a stochastic geometry dynamics inspired by opinion dynamics that captures the essence of modern asymmetric social networks with leaders and followers. Points in the Euclidean space represent opinions, and the leader of an agent is the one with the closest opinion. In this dynamics, each follower updates its opinion by halving the distance to its leader. We demonstrate that this simple dynamics and its iterations exhibit several interesting purely geometric phenomena related to the evolution of leadership and opinion clusters, which resemble those observed in social networks. We also show that when the initial opinions are randomly distributed as a stationary Poisson point process, the spatial frequency of each of these phenomena can be expressed through an integral geometry formula involving semi-algebraic domains. Finally, we analyze numerically the limiting behavior of this follower dynamics. In the Poisson case, the agents fall into two categories: ultimate followers, who continue updating their opinions indefinitely, and ultimate leaders, who adopt a fixed opinion after a finite time. Spatial discrete event simulations support all our findings.
\end{abstract}

\section{Introduction}

People's opinions and beliefs are influenced in complex ways by families, friends, colleagues, and media, as well
as by politicians and other mega-influencers \cite{Acemoglu_opinion,Bisin_2000,Bisin_2001,haensch2023,Boyd_1985,Cavalli_1981}.
In recent decades, attempts have been made to understand aspects of this process using mathematical modeling and computational simulation such as surveys on opinion dynamics
  \cite{Aydogdu_2017,Anderson_2019,Lorenz_2007,Mossel_2017,Proskurnikov_2017,Proskurnikov_2018}.

Many models of opinion dynamics are based on the assumption  that we are influenced more easily by people with whom we {\em almost} agree than by those whose views starkly differ from ours.
A similar but more general phenomenon, known as {\em biased assimilation}, is recognized among psychologists; referring to our tendency  to filter and interpret information in a way that supports our preconceived notions \cite{Lord_1979}. Models of opinion dynamics built on this assumption are known as {\em bounded confidence models} \cite{Aydogdu_2017,Canuto_2012,Mirtabatabaei_2012}.  A popular example is the Hegselmann and Krause model \cite{hegsel_krause_2002,Hegselmann_Krause_2005}, which builds on earlier work by Krause \cite{krause_1997,krause_2000}. It has been extensively studied in the literature \cite{Lorenz_2005, Lorenz_2006, Blondel_2009,Perrier_2024}, and will serve as our starting point here. The Hegselmann-Krause model examines the stochastic time evolution of a voter's opinion in response to the opinion of other like-minded voters.

The original Hegselmann-Krause model is discrete
in both time and  opinion space with a finite number of agents moving one step
at a time. Similar models that are
continuous in time \cite{Piccoli_2021}, 
opinion space \cite{Wedin_2015}, or opinion space and time \cite{Borgers_candidate_voter,Goddard_2022} have been proposed since then.

This paper proposes a mathematical model allowing one to analyze both short and long term opinion dynamics
of a countable collection of agents. At any time step of the dynamics, each agent has an
$\mathbb{R}^d$-valued opinion and can only
be influenced by the agent whose opinion is closest to its opinion, which is referred to as 
its \textit{leader}.
This new model is meant to capture the essence of such social network interactions
as Instagram and Twitter, where the asymmetric leader/follower dynamics is central.

The follower/leader connections form a directed graph,
where the direction goes from the follower towards the leader. 
Two agents which follow each other will be called a \textit{leader pair}. This structure defines a directed graph with different connected components. We will also show that, in this Poisson model, there are no cycles longer than 2 a.s. 
We think of a group of agents that are part of the same connected component
of this follower graph as a \textit{party}.

If the initial opinions form a random point process, the directed graph is generated over a random point set,
and is hence a random graph.
In several parts of this work, we will restrict ourselves to the case where the points of the initial configuration are
sampled according to a stationary Poisson point process over $\mathbb{R}^d$. 
This choice is dictated by tractability reasons
(other initial conditions will be considered in future steps). For instance, in this Poisson case, the
follower/leader random graph is called
the \textit{Poisson Nearest Neighbor graph} and
parties are called \textit{Poisson descending trees}.
They have been studied extensively \cite{nearest_neigh,haenggi}.
An example of such a tree is depicted in Figure \ref{fig:example} for $d=2$.

\begin{figure}
\centering
  \includegraphics[width=2.6in]{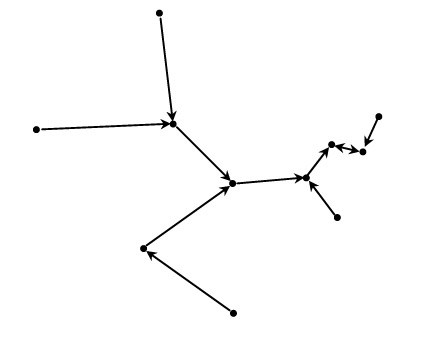}
  \caption{Example of a Poisson descending tree }
  \label{fig:example}
\end{figure}

The initial random graph evolves over time based on the prescribed dynamics.
More precisely, given an initial point process, 
at the first time step of the dynamics, each agent's opinion moves halfway to that of its leader, which will be shown to be uniquely defined in the Poisson model.
This gives a second point process on which a new 
follower/leader directed graph is defined. The dynamics is then applied to this new graph, which
defines a third point process, and so on.
The \textit{Poisson Follower Point Process of order $n$} is the point process obtained
at the $n^{th}$ iteration of this dynamics when starting from a Poisson Point Process.
At the $n^{th}$ iteration, the agents or points are the vertices of a graph that will be referred to as
the \textit{Poisson Follower graph of order $n$}. 
As we will see, interestingly a given agent may follow different leaders at different times, i.e., the structure of the Poisson
follower graph changes with $n$. 

This dynamics can be seen defined in contrast to the Hegselmann-Krause dynamics \cite{hegsel_krause_2002}.
The main difference is that opinion updates consist in averaging each opinion with 
that of its leader (in contrast to averaging each opinion over the set of opinions in a 
ball around it as in the Hegselmann-Krause dynamics). Also, the dynamics is 
extended to a particle system setting \cite{liggett_particle_systems}, namely to a stationary
point process of $\mathbb R^d$ (in contrast to the Hegselmann-Krause dynamics which is only discussed
on compact domains so far). 


One of the surprising facts is that this simple dynamics, purely determined by the relative
positions of agents, leads to a large collection of phenomena to which one can associate natural names
in view of what is observed in social networks.
Here are some instances of the phenomena observed in this dynamics in the Poisson case
and that are discussed in detail in Section \ref{definitions}:

\begin{itemize}
\item One agent can be a \textit{leader to several agents}. The maximal number depends on the dimension.
In dimension two, for all point configurations, a point can be a leader for up to 6 followers.
This number is at most 5 a.s. in the Poisson case. This is related to the kissing number
\cite{kissing_number}. 

\item Each initial party has a.s. exactly one "loop", i.e., only one pair of mutually closest 
neighbors, or equivalently a leader pair. 

\item If two points form a leader pair, then, after one step, their opinions merge 
and will never change.

\item At each step of the dynamics, an agent either keeps the same leader or follows another agent.
Similarly, a leader at any step can \textit{keep its followers, gain, or lose some followers}.

\item Surprisingly, two agents
in a follower/leader pair can swap their roles in one step of the dynamics. 
This \textit{follower inversion} is reminiscent of the mentor/mentored
scenario where the direction of leadership changes as time evolves.

\item At any step of the dynamics, agents can be arranged in connected components called parties bellow; finite parties can be characterized by their leader pairs.

\item As a result of this dynamics, at each step, agents either stay in their party or switch positions
within a party, or change party. Equivalently, under a step of the dynamics, \textit{a party can either
stay the same, split, loose, or gain new agents}. 


\end{itemize}

The identification of these purely geometry based dynamical phenomena is 
complemented by analytical and simulation results on the Poisson model. The paper starts with
Section 2 where we give the notation and the point process background
used in the rest of the document. Section 3 presents the dynamical phenomena. It is one of the main contributions of the paper.

The main probabilistic question studied in the paper is about the spatial frequency
of the phenomena described above at any given step. 
In Section~\ref{sc:frequency_phenomena}, we show that in the Poisson case,
the frequencies in question admit representations in terms
of integral geometry formulas involving integration domains which are semi-algebraic sets
of the Euclidean space and integrands which are determined by the volume of certain unions of balls.
Section~\ref{numerical_methods} is a complement to 
Section~\ref{sc:frequency_phenomena} which discusses
numerical methods to evaluate the integral geometry
formulas alluded to above and spatial simulation methods for the complementary estimation of the frequencies in question.


The next question discussed in the paper is of dynamical system nature and concerns the limiting behavior of the 
Poisson follower point process of order $n$ and of the Poisson follower graph of order $n$,
as $n$ tends to infinity. These questions are discussed in Section~\ref{sc:asymptotic}.
We conjecture a classification of the points of the initial Poisson point
process into two categories: ultimate followers, which remain followers forever and whose
opinions only converge weakly to their limiting value, and utimate leaders
whose opinion converges in total variation to their limiting value, which is hence
reached in a finite number of steps.  In the same section, we study stable trees, i.e., trees that do not change their structure with time.  

\section{The Follower Model and its Background}\label{sc:notation}

\subsection{Overview of Point Process }

Let $|| \bullet ||$  denote the Euclidean norm on $\mathbb{R}^2$ and $B(A,r)$ the open ball of radius $r$ and center $A$. 

Let $S$ be the $\mathbb{R}^2$ space equipped with its Borel $\sigma$-
algebra $\mathcal{B}$ and Lebesgue measure $\nu$. A point process $\Phi$ on $S$ is a random locally
finite subset of $\Lambda$. One can also view $\Phi$ as a random counting measure on $\Lambda$, having the form $\Phi=\sum_{i\in \mathbb{N}}\delta_{X_i}$, where $\{X_i\}_{i\in \mathbb{N}}$ is a countable collection of points in $S$ with no accumulation points.

A point process is called \textit{simple} if $\Phi(\{x\})\leq 1$ for all $x \in S$. A point process whose distribution is invariant under translations is called \textit{stationary}. The \textit{intensity measure} of a point process $\Phi$ is the measure on $S$ defined by
\[\lambda(B)=\mathbb{E}[\Phi(B)], \, \, B\in \mathcal{B}(\Lambda).\]

\begin{definition}
For any counting measure $\mu=\sum_{i\in\mathbb{N}} \delta_{x_i}$ and $n\in \mathbb{N}$, its $n$-th power in the sense of products of measures is 
\[\mu^n=\sum_{(i_1,...,i_n)\in \mathbb{N}^n}\delta_{(x_{i_1},...,x_{i_n})}.\]
Define the $n$-th factorial moment measure as the following counting measure on $(\mathbb{R}^2)^n$.
\[\mu^{(n)}=\sum_{(i_1\neq...\neq i_n)}\delta_{(x_{i_1},...,x_{i_n})}.\]
\end{definition}

 \begin{definition}{Moment measures.}
For a point process $\Phi$ on $\mathbb{R}^2$, let $\Phi^n$ be the $n$-th power of $\Phi$ and $\Phi^{(n)}$ be the $n$-th factorial moment of $\Phi$. We call $M_{\Phi^n}=E[\Phi^n(B)]$ the $n$-th moment measure (the first moment measure is the mean measure) of $\Phi$ and $M_{\Phi^{(n)}}=E[\Phi^{(n)}(B)]$ the $n$-th factorial moment measure.
\end{definition}

Moment measures provide important average structural properties of the process, such as level of clustering or repulsion. 

\begin{definition}(Poisson point process)
Let $\Lambda$ be a locally finite measure on l.c.s.h. space $\mathbb{G}$. A point process $\Phi$ is said to be Poisson with intensity measure $\Lambda$ if for all pairwise disjoint sets $B_1, ..., B_j \in \mathcal{B}(\mathbb{G})$, the random variables $\Phi(B_1),..., \Phi(B_j)$ are independent Poisson random variables with respective means $\Lambda(B_1),..., \Lambda(B_j)$; i.e. $\forall m_1,... ,m_j \in \mathbb{N}$,  
\begin{equation}
\mathbb{P}(\Phi(B_1)=m_1,..., \Phi(B_j)=m_j)=\prod_{i=1}^j \frac{\Lambda(B_i)^{m_i}}{m_i!}e^{-\Lambda(B_i)}.
\end{equation} 
\end{definition}

\begin{definition}(Homogeneous Poisson point process on $\mathbb{R}^2$)
If $\Phi$ is a Poisson point process on $\mathbb{R}^2$ with intensity measure $\Lambda(dx)=\lambda \times dx$ where $\lambda \in \mathbb{R}^{*}_{+}$ and $dx$ denotes the Lebesgue measure, then $\Phi$ is called a homogeneous Poisson point process of intensity $\lambda$.  
\end{definition}


Throughout the paper, we denote by $\Phi=\sum_{i\in \mathbb{N}}\delta_{x_i}$ the homogeneous Poisson point process on $\mathbb{R}^2$ with intensity $\lambda$, which serves as initial condition to the dynamics. For all $n\geq 0$, the Poisson Follower point process of order $n$ will be denoted by $\Phi_n$. Note that $\Phi=\Phi_0$. Let $\mathbb{R}^{*}=\mathbb{R}\setminus \{0\}$.

For two points $x, y \in \Phi$, we denote by $B(x\rightarrow y)$ the open ball with center $x$ and radius $d(x,y)$, which is the distance between $x$ and $y$. To indicate the distance between two points $x$ and $y$ at the $n^{th}$ iteration, we will write $d(x,y, \Phi_n)$.

\paragraph{Palm Theory}
Informally, the Palm measures of a point process $\Phi$ at a point $x\in \Lambda$ is the probability measure of $\Phi$ conditioned on having a point at location $x$. For a more detailed discussion on the matter, see \cite{baccelli_nova_knjiga,kallenberg}.

\subsection{Notation for the Follower Dynamics}
Consider a counting measure on  $\mathbb{R}^2$. Take two points $A$ and $B$ in this counting measure.
If the closest point to $A$ is $B$, we say that $A$ \textit{follows} $B$. Other equivalent statements are,
"$A$ is a \textit{follower} of $B$", or "$B$ is a \textit{leader} of $A$", and "$B$ is \textit{followed} by $A$", etc.
We will also use notation $A \rightarrow B$ to mean that $A$ follows $B$. If $B$ follows $C$, and $A$ follows $B$, 
we call $C$ a \textit{leader of order two} of $A$. 
We call \textit{leader pair} a pair of agents that follow each other. In a homogeneous Poisson point process, each point A has a.s. a unique leader and there are a.s. no other cycles involved than leader pairs \cite{nearest_neigh}. 

Because of scale invariance of the Poisson point process, we will assume without loss of generality 
that $\lambda=1$ in subsequent sections.

\section{Phenomena Observed in the Follower Dynamics}\label{definitions}

In this section, we give some basic definitions to be used throughout the paper.
We also identify and set the terminology for the phenomena of interest
and give some illustrations.

\subsection{Agent Events}
Here is a list of situations
that can happen to an agent when applying the dynamics at any step. 

\paragraph{Leader swap of an agent}
The "out" edge of this agent is either kept  (\textit{leader keep}) or
changed ($\textit{leader swap}$) by the dynamics. Notice that these two situations
are mutually exclusive since, by construction, each point gets only one leader at any
given time. For the examples in Figure \ref{fig:ex_follower_loss}, Agent 3 swaps
its leader from 1 to 2, as a result of the dynamics.

\begin{figure}
\begin{minipage}{.5\textwidth}

\centering
  \includegraphics[width=2in]{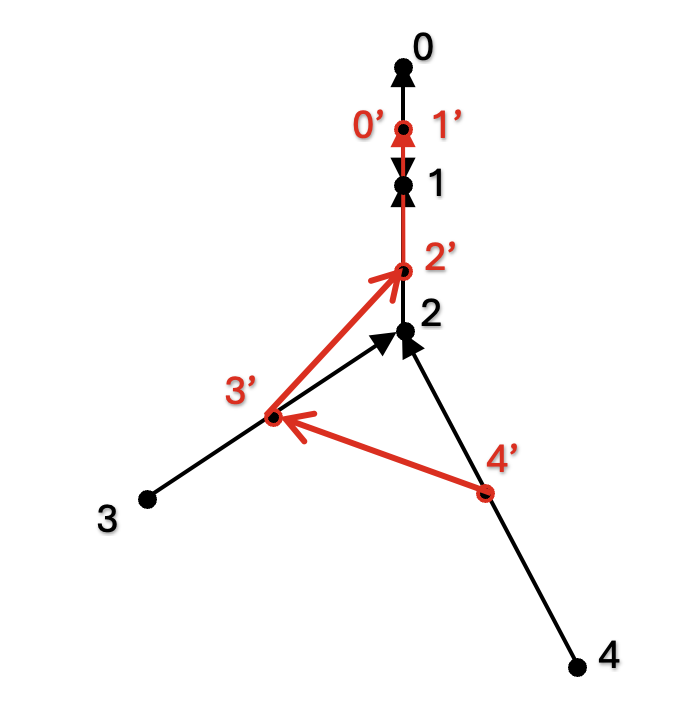}
\end{minipage}
\begin{minipage}{.5\textwidth}
\centering
  \includegraphics[width=2in]{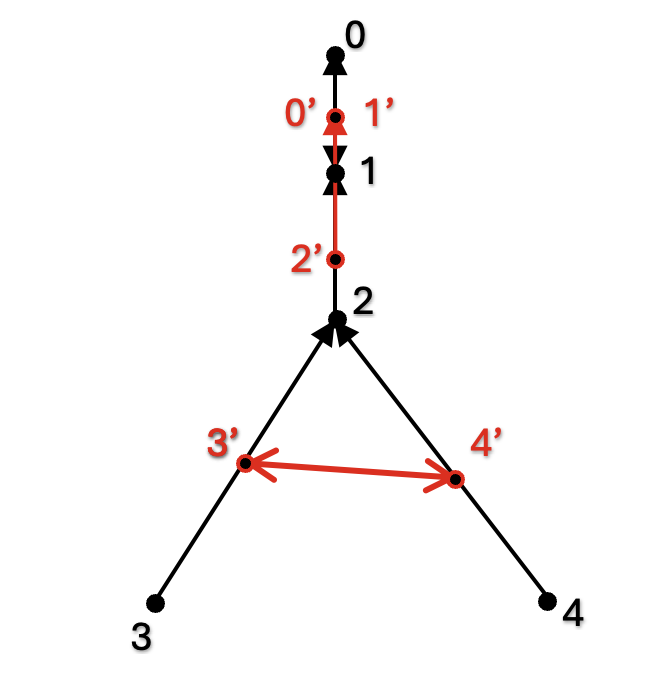} 
  \label{fig:ex_fusion}
  \end{minipage}
  \caption{The Step 0 positions are shown in black. Step 1 positions are in red, with the new
	positions denoted with the prime '. Note that 0 and 1 become 1' in both images.
	\textbf{ Left:} Example of a Follower Loss/Follower Gain. \textbf{Right:} 
	Example of formation of a leader pair of order 1. Same convention 
	as in the left image. In both cases 4 experiences a leader swap whereas 2 experiences a leader keep.}
  \label{fig:ex_follower_loss}
\end{figure}

\paragraph{Follower gain, loss or keep of an agent} 
Concerning the set of "in" edges of the agent, one can have a \textit{follower loss},
a \textit{follower gain} or a \textit{follower keep}. The last case is that where
the set of followers is the same before and after the considered step of the dynamics.
Notice that 
a loss and a gain can happen at the same time.
In the left of Figure \ref{fig:ex_follower_loss}, Agent 1 lost its follower,
Agent 3; Agent 2 gained a follower, namely Agent 3. 

\paragraph{Leader pair} We call initial leader pairs, agents involved in a leader
pair of order 0 (i.e. at step 0). New leader pairs, forming at step 1. will be called 
leader pairs of order 1, etc.
 
\paragraph{Leader pair formation by an agent} This situation is that where an agent 
which is not in a leader pair becomes part of a leader pair at the next step 
of the dynamics. An example of this is
shown in the right Figure \ref{fig:ex_follower_loss}. 

\begin{figure}
\begin{minipage}{.4\textwidth}
\centering
  \includegraphics[width=2in]{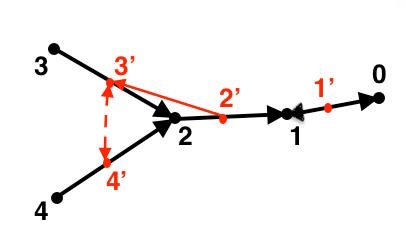}
  \caption{Example of follower inversion. Again 0 and 1 merge to become 1'. The positions at the step 1 are shown in red, with the new positions denoted with the prime, like for example $1'$. }
  \label{fig:ex_follower_inversion}
  \end{minipage}
\begin{minipage}{.7\textwidth}
\centering
  \includegraphics[width=2.5 in]{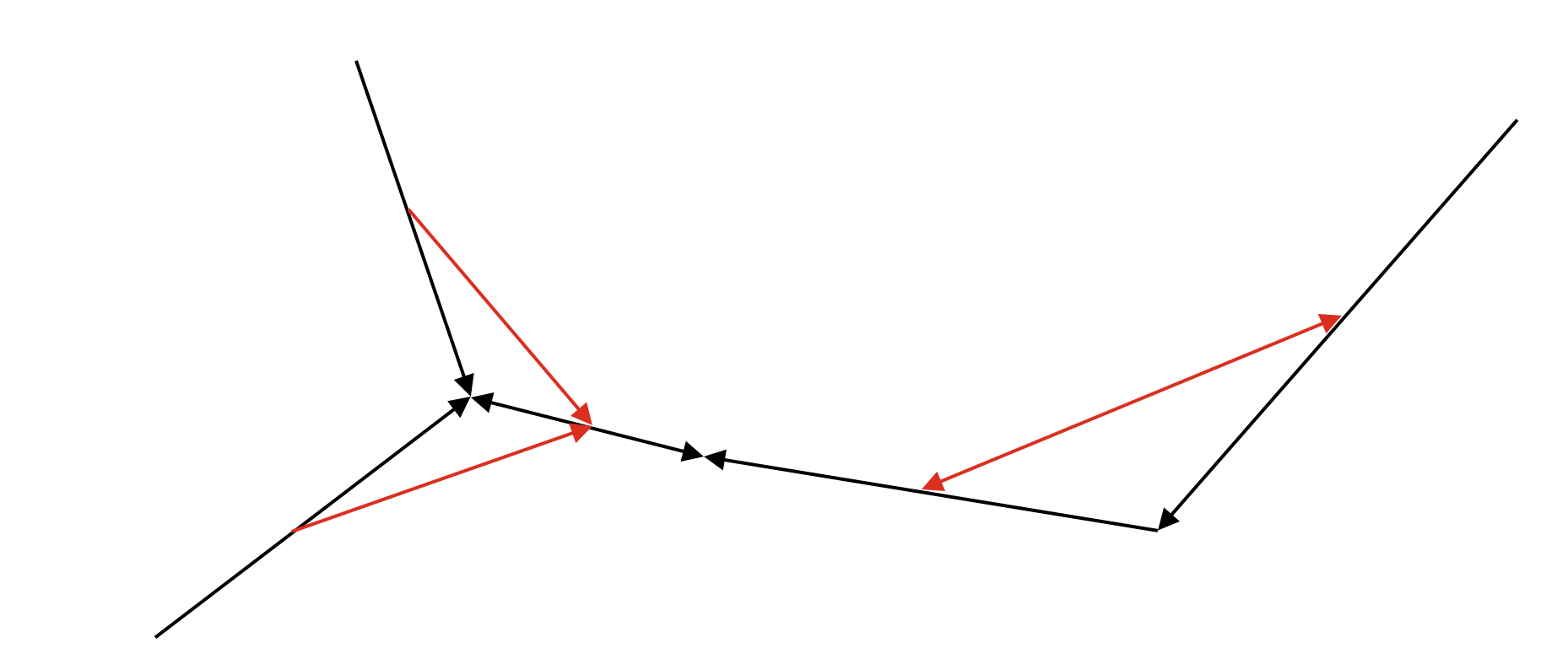} 
  \captionof{figure}{ Example of party fission}
  \label{fig:ex_step1}
  \end{minipage} 
\end{figure}

\paragraph{Follower inversion for an agent} Another special case is that where, after one step,
the leader of an agent becomes its follower. We call such a situation a \textit{follower inversion}.
We do not count here situations where the leader and follower become a leader pair at step 1
as a follower inversion. An example of follower inversion is shown in Figure \ref{fig:ex_follower_inversion}. Agent 3 was initially following Agent 2, but at step 1,
Agent 2 follows Agent 3, whereas Agents 3 and 4 form a leader pair.

\subsection{Parties}

\paragraph{Forward and Backward Sets} We denote by $\mbox{For}(x,\Phi)$ the leader
set of all orders of $x$ in $\Phi$, and we will call it the \textit{forward set of $x$}.
$\mbox{Back}(x,\Phi)$ denotes the follower set of all orders of $x$ in $\Phi$, and we call
it the \textit{backward set of $x$}. For each agent, there is a forward and a backward set
at each step of the dynamics. See Figure \ref{fig:forward_backwardsets} for these sets at step 0. 
\begin{figure}[h!]
\centering
  \includegraphics[width=2.5in]{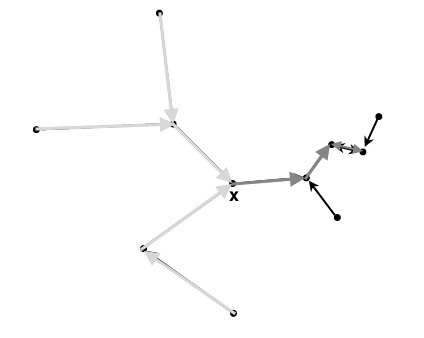}
  \caption{Representation of the forward and backward sets of an agent. The forward set of $x$
	is connected to $x$ with dark gray arrows while the backward set of $x$ is with light gray ones.
	The agents connected by black arrows are in the party of $x$ but neither in the backward
	set nor in the forward set of $x$.}
  \label{fig:forward_backwardsets}
\end{figure}

\paragraph{Parties} 
Roughly, we call \textit{party} any connected component of the follower graph at any step.
Note that two points $x$ and $y$ belong to the same party iff there exists a point $z$
(not necessarily different from $x$ or $y$), such that both $x$ and $y$ belong to the backward set of $z$,
i.e., $x, y \in \mbox{Back}(z,\Phi)$.

\paragraph{Parties at step $k$} are defined as connected components at step $k$. We conjecture that all parties are a.s. finite at all steps.
At any step, finite parties are characterized by their leader pair.
Since a leader pair cannot be dissolved,
the number of parties can only increase with the number of steps.
An example can be found in Figure \ref{fig:forward_backwardsets}. 
Since the initial agent distribution is Poisson at step 0, every party is almost surely
finite at step 0 \cite{nearest_neigh}. 

\paragraph{Party Gain/ Party Loss}
A given party (once it exists) can have fluctuations in its
cardinality.
A \textit{party gain} is an increase in the party's cardinality at a given step; this requires 
that at least
one agent of another party joins.
A \textit{party loss} is a decrease in its cardinality at a given step; this requires at least
one agent of this party to leave.

\paragraph{Party fission} The situation in which a new leader pair is created
within a party is called a \textit{party fission}. The initial party survives (its leader pair cannot
be dissolved). A new party is however formed, associated with this new leader pair.
Figure \ref{fig:ex_step1} gives an example of such a fission.

\paragraph{Party swap of an agent}
A given agent can belong to different parties (as defined above)
at different steps. For a given agent, a \textit{party swap} is a change
of party (namely a change of leader pair) at a given step.
An example of party swap is shown in right of Figure~\ref{fig:fussion_swap},
where agent $z$ and its followers are subject to a party swap. 
Note that an agent may experience a party swap even if it keeps its leader, but this leader
is itself subject to a party swap (as for agent $x$ in this example).
Another example of party swap of an agent is shown in left of Figure ~\ref{fig:fussion_swap} where
Agent A is subject to a party swap.

\paragraph{Party restructuring} By \textit{party restructuring}, we mean a situation where
the dynamics leads to a change within the directed graph describing the party, 
but to no change in the set of agents contained in the party, like the situation illustrated
in left of Figure \ref{fig:ex_follower_loss}. 

\paragraph{Party stability} A party is stable if when applying the dynamics, the set of agents
involved is unchanged and the structure of the directed graph describing the party is unchanged.

\paragraph{Indefinitely stable party} A party is indefinitely stable if it is stable
at any future time step. Stable parties exist (e.g. any linear directed path ending
with a leader pair) and are discussed in Section \ref{stable_chain}.

\begin{figure}[h!]
\centering
\begin{minipage}{.24\textwidth}
  \includegraphics[width=2.8in]{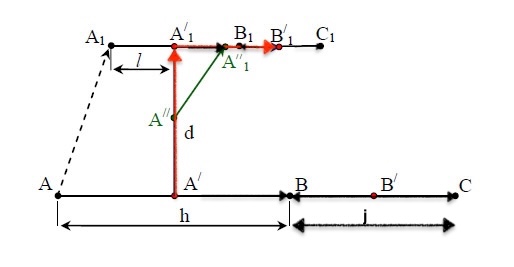}
  \end{minipage}
  \begin{minipage}{.75\textwidth}
  \centering
  \includegraphics[width=2.in]{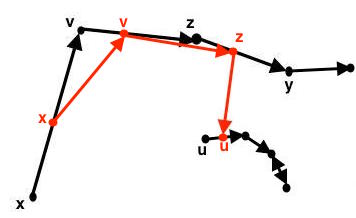}
  \end{minipage}
  \caption{\textbf{Left:} Example of a 4 body swap. A',B', $A_1'$ and $B_1'$ are the positions
	of agents in the next time step. Positions in the second time step are denoted by $A"$
	and $A_1"$ and connections are shown in green. \textbf{Right:} Example of a party swap.
	Initial positions are shown in black and step 1 is in red. Agent $z$ swaps leaders and
	parties. Together with $z$, agents $x$ and $v$ also experience a party swap.}
  \label{fig:fussion_swap}
\end{figure}

\paragraph{4 body party swap} 
A \textit{4 body party swap} is an agent party swap that involves 4 different agents:
$ A,B,A_1,$ and $B_1$, with, at step $k$, $A$ following $B$ and belonging to a party, 
and $A_1$ following $B_1$ and belonging to another party.
Then, at step $k+1$, $A$ follows $A_1$ which follows $B_1$, and
neither $B_1$ nor $B$ swap party. Hence $A$ is subject to a party swap.
For the example of the left of Figure~\ref{fig:fussion_swap}, 
for this party swap of $A$ to take place,
the following inequalities must hold:
$h^2<(\frac{h}{2}-l)^2+d^2$ ($A$ follows $B$ at step $k$), 
$d<\frac{h+j} 2$ ($A$ follows $A_1$ at step $k+1$). It is easy to check 
that there are solutions to these inequalities. In this figure, $(B,C)$
forms a leader pair as well as $(B_1,C_1)$, so that this scenario is a party
swap for $A$.

\begin{figure}[h!]
\centering
\begin{minipage}{.24\textwidth}
  \includegraphics[width=2.1in]{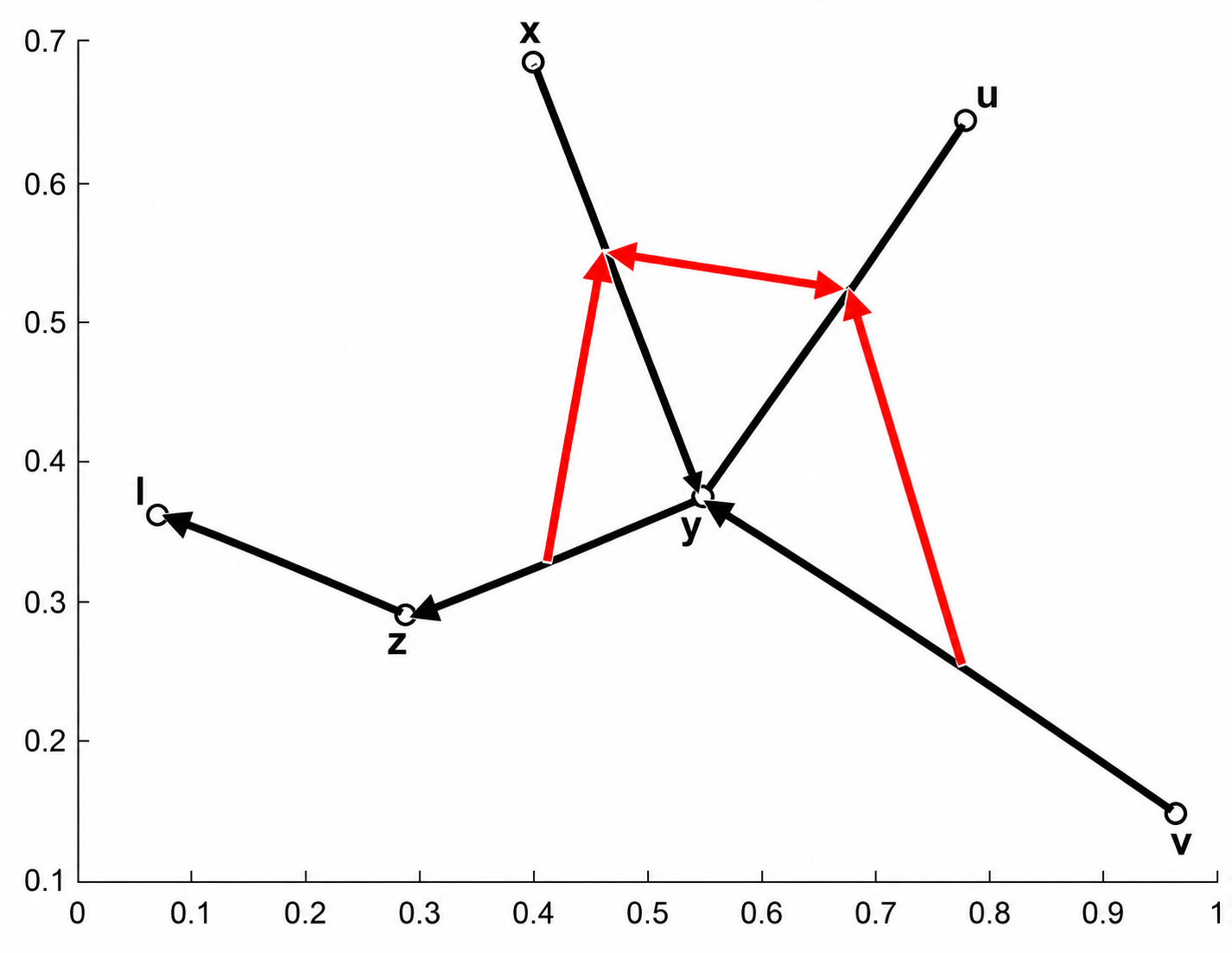}
  \end{minipage}
  \begin{minipage}{.75\textwidth}
  \centering
  \includegraphics[width=2.in]{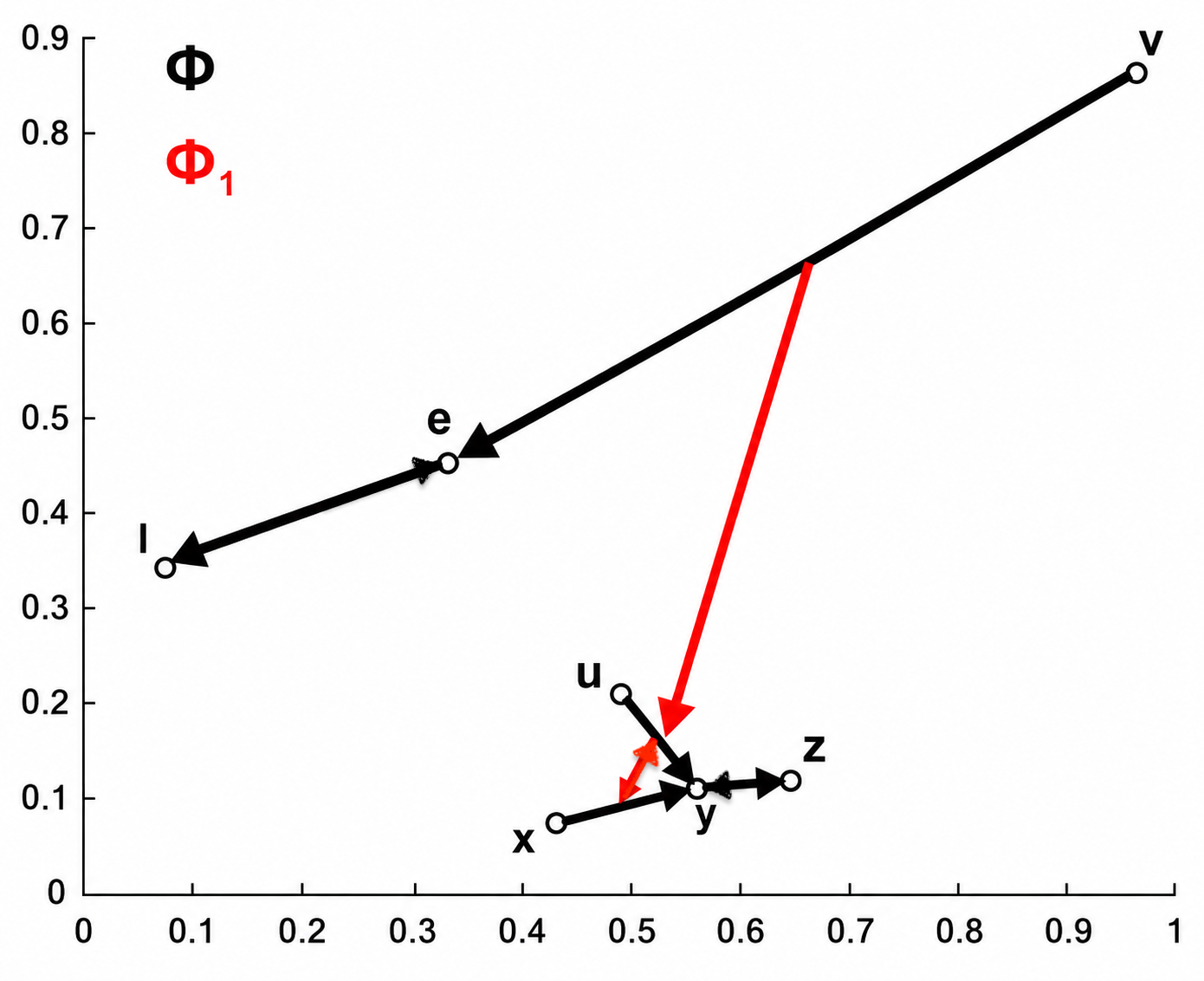}
  \end{minipage}
  \caption{Step 0 positions and connections are shown in black. Step 1 connections are shown in \textcolor{red}{red}. \textbf{Left:}  At step 0, $x$ follows $y$ whereas  at step 1,
  $y$ follows $x$. We hence have a follower inversion. In addition, $x$ and $u$ become a leader pair of order 1. \textbf{Right:} Agent $v$ swaps parties and agents $u$ and $x$ form a leader pair of order 1.}\label{fig:examples_different_phenomena}
\end{figure}

\paragraph{Remark} Under the Follower Dynamics,  in the Poisson case, we observe all the phenomena described in this section, and it is possible that after one step several events happen at the same time on a given set of agents. For example, at the agent level, in the left Figure \ref{fig:examples_different_phenomena}, we observe both a follower inversion and the creation of a leader pair of order 1. 
At the agent level, in the right Figure \ref{fig:examples_different_phenomena}, we observe both a party swap and the formation of a leader pair of order 1. The right of Figure \ref{fig:examples_different_phenomena} is also an example of a simultaneous party restructuring and party swap.

\section{Frequency of the Phenomena}\label{sc:frequency_phenomena}

In this Section, we develop a systematic method to calculate or estimate exact values or bounds on the spatial frequencies of related events using integral geometry. We demonstrate that this involves evaluating integrals of specific exponential functions over semi-algebraic sets. Semi-algebraic sets are subsets of $\mathbb{R}^2$, which are finite unions of solutions to polynomial equations and inequalities with coefficients in $\mathbb{R}$ \cite{mohab}. This approach applies to each of the phenomena listed above.


\subsection{Integral Geometry Estimates of Densities of Leader Pairs}\label{integral_geom_leader_pairs}

In this subsection, we give an integral geometry representation of leader pairs and then give numerical estimates. We illustrate the method of calculating the densities by showing a step by step calculation of the density of leader pairs in the initial configuration, i.e., the density of order 0.  One can calculate frequencies of other configurations in the same manner.  Then, we calculate an upper bound and an exact integral geometry formula
for the density of leaders of order 1, type 1 or \textit{density of order one type 1}. This is a known result \cite{haenggi}, but we derive it here in a new way that can be extended to other phenomena.  We also show how to calculate the density of leader pairs of order one.
Again, by leader pairs, we mean a nearest neighbor pair. 

\subsubsection{Order Zero}

Consider the point process $N^{(2)}$ consisting of pairs of points of $\Phi$ which are  mutually closest points. Namely 
 \[N^{(2)}=\sum_{i\neq j \in \mathbb{N}}\delta_{x_i,x_j}1_{\Phi(B(x_i \rightarrow x_j))=1},\\ 1_{\Phi(B(x_j \rightarrow x_i))=1 },\]
where $\Phi(B(x_j \rightarrow x_i))$ denotes the number of points in the open ball centered at $x_j$ with the radius $d(x_j,x_i)$.

The mean measure of $N^{(2)}$ is given by, for $A\subset \mathbb{R}^2\times \mathbb{R}^2$, 
\begin{equation}
\begin{split}
\mathbb{E}[N^{(2)}(A)]=\mathbb{E}[\int_{\mathbb{R}^2\times \mathbb{R}^2}1_{x,y\in A}1_{x\neq y}  1_{\Phi(B(x \rightarrow y))=1}
1_{\Phi(B(y \rightarrow x))=1 }\Phi^{(2)}({\rm d}x\times {\rm d}y)],
\end{split}
\end{equation}
where $\Phi^{(2)}$ is the Poisson factorial moment measure of order 2 (\cite{baccelli_nova_knjiga} Section 3.3.2).
By the higher order Campbell-Little-Mecke formula \cite{baccelli_blas},

\begin{equation}
\begin{split}
 \mathbb{E}[N^{(2)}(A)]=\int_{\mathbb{R}^2\times\mathbb{R}^2}1_{x,y\in A}\mathbb{E}^{x,y}[ 1_{\Phi(B(x\rightarrow y))=1} 1_{\Phi(B(y \rightarrow x))=1}]\lambda^{(2)}({\rm d}x {\rm d}y), 
\end{split}
\end{equation}
where $\mathbb{E}^{x,y}$ is the two point Palm expectation and $\lambda^{(2)}$ is the factorial Poisson moment measure of order 2 (\cite{baccelli_nova_knjiga} Section 3.3.2). Recall that for a Poisson point process, the $n^{th}$ factorial moment measure equals the $n^{th}$ power of the intensity measure. We will be using this over and over again in the calculations that follow.

Now by Slivnyak's theorem \cite{baccelli_blas}
\begin{equation}
 \mathbb{E}[N^{(2)}(A)]
= \int 1_{x,y\in A}
\mathbb{E}[1_{\Phi'(B(x\to y))=1} \, 1_{\Phi'(B(y\to x))=1}]
\lambda^2 {\rm d}x {\rm d}y
\end{equation}
where $\Phi'=\Phi + \delta_x+\delta_y$. Observe that $B(x\to y)$ is the open ball centered at $x$ with radius
$d(x,y)$. Hence $x\in B(x\to y)$ whereas $y\notin B(x\to y)$.
Similarly, $y\in B(y\to x)$ whereas $x\notin B(y\to x)$.
Therefore, under the Palm version $\Phi'=\Phi+\delta_x+\delta_y$,
the event $\{\Phi'(B(x\to y))=1\}$ means that the only point of
$\Phi'$ contained in $B(x\to y)$ is the deterministic Palm point $x$.
Equivalently, the underlying Poisson process $\Phi$ has no points in
$B(x\to y)$. The same argument applies to $B(y\to x)$.
Using the change of variables $(x,y)\rightarrow (x,u)$ with $u=y-x$ and taking $A=C\times \mathbb{R}^2$, with $C \in \mathcal{B}(\mathbb{R}^2)$, we get
\begin{equation}
\mathbb{E}[N^{(2)}(C\times \mathbb{R}^2)]=\int_{\mathbb{R}^2}\int_{\mathbb{R}^2}1_{x \in C}1_{u\in \mathbb{R}^2} \mathbb{E}[1_{\Phi'\circ \theta_x(B(o\rightarrow u))=1} 1_{\Phi'\circ \theta_x(B(u \rightarrow 0))=1}]\\
 \lambda^2 {\rm d}x {\rm d}u,
\end{equation}
 where $\theta_x$ is a shift operator. Since the translated Palm process contains the deterministic points
$0$ and $u$, and since $0\in B(o\to u)$ and $u\in B(u\to 0)$,
the conditions
$\Phi'\circ\theta_x(B(o\to u))=1$ and
$\Phi'\circ\theta_x(B(u\to 0))=1$
are equivalent to the absence of further Poisson points in the
corresponding balls, namely
$\Phi(B(o\to u))=0$ and $\Phi(B(u\to 0))=0$ \cite{baccelli_blas}.
The expectation is the probability for point $u$ to be the closest to $0$ and for $0$ to be closest to $u$. Then
\begin{equation}
\mathbb{E}[N^{(2)}(C\times \mathbb{R}^2)]=|C|\int_{\mathbb{R}^2} 1_{u \in \mathbb{R}^2} \mathbb{E}[1_{\Phi(B(o\rightarrow u))=0} 1_{\Phi(B(u \rightarrow 0))=0}]\\
 \lambda^2{\rm d}u.
\end{equation}
Because of symmetry, and after switching to polar coordinates, we get:
\begin{equation}
\mathbb{E}[N^{(2)}(C\times \mathbb{R}^2)]=|C|\int_{\theta\in [0, 2\pi]}\int_{0}^\infty v e^{-\lambda v^2\pi}e^{-\lambda v^2(\pi/3+\sqrt{3}/2)}
 \lambda^2{\rm d}v. 
\end{equation}
Hence
\begin{equation}
\mathbb{E}[N^{(2)}(C\times \mathbb{R}^2)]=|C| \int_0^\infty 2\pi v e^{-\lambda v^2\pi}e^{-\lambda v^2(\pi/3+\sqrt{3}/2)}
 \lambda^2{\rm d}v. 
\end{equation}
To evaluate the integral, we use the change of variables $v\rightarrow r$ with $r=\pi v^2$, and recall that $\lambda=1$. Then we get 
\[\mathbb{E}[N^{(2)}(C\times \mathbb{R}^2)])=|C|\int_0^\infty  e^{-r(\frac{\pi+(\pi/3+\sqrt{3}/2)}{\pi})}{\rm d}r=|C|\cdot \frac{\pi}{\pi+\pi/3+\sqrt{3}/2}\approx 0.62|C|.\]

So the density of the ultimate leader pairs of order 0, i.e. points that after one step stay together forever, is $0.62$. Note that we are double counting the points since they are two points at one spot.

\subsubsection{Density of Leader Pairs of Order One}\label{density_order1}

In this subsection, we continue with similar calculations to determine the intensity of certain leader pairs of order $1$. First, we provide an upper bound on this density and then explain how to find the exact density. Each calculation requires the positions of at least $4$ points to know the positions at step 1, see Figure \ref{fig:ex_2_swap}. The first situation involves two followers of one point that become a leader pair at step 1. An instance is given on the left of Figure \ref{fig:ex_2_swap}. The second involves cases where a leader and a follower become a leader pair at the next step. An instance can be found on the right of Figure \ref{fig:ex_2_swap}. Densities of both types are computed similarly, so we focus on the detailed calculation of type 1 only. 
\begin{figure}
\begin{minipage}{.5 \textwidth}
  \centering
  \includegraphics[width=.6\linewidth]{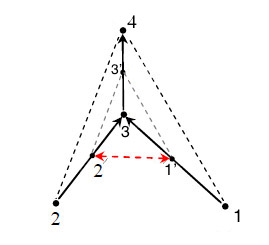}
\end{minipage}
\begin{minipage}{.5\textwidth}
\centering
  \includegraphics[width=2.4in]{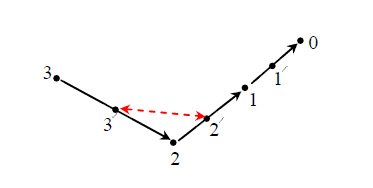}
\end{minipage}
  \caption{\textbf{Left:} Four points and their positions at step 1. Example of leader pair of order 1, type 1: two followers of one point become a leader pair at step 1. \textbf{Right}: Example of leader pair of order 1, type 2. A point and its leader become a leader pair at step 1.}
  \label{fig:ex_2_swap}
\end{figure}

\paragraph{Type 1}  As mentioned above, leader pairs of order 1, type 1 are formed from the configurations where two step 0 followers of one agent become a leader pair at step 1. We first give an upper bound on the frequency of type 1 and then we analyze the exact frequency.

Let $\mathcal{D}\in {\mathcal{B}}^4$ be the set of all distinct points
$z_1$, $z_2$, $z_3$, and $z_4$ that satisfy the conditions (\ref{cond12})-(\ref{cond15}) defined below:

\begin{equation}\label{cond12}
d(z_1,z_3)<d(z_1,z_2), \, d(z_1,z_3)<d(z_1,z_4),
\end{equation}
which are the conditions needed for $z_1$ to follow $z_3$, in the absence of other points than these four points;
\begin{equation}\label{cond13}
d(z_2,z_3)<d(z_2,z_1), \, d(z_2,z_3)<d(z_2,z_4),
\end{equation}
which are the conditions needed for $z_2$ to follow $z_3$ as well, also in the absence of other points. Finally
\begin{equation}\label{cond14}
d(z_3,z_4)<d(z_3,z_1), \, d(z_3,z_4)<d(z_3,z_2),
\end{equation}
which are the conditions needed for $z_3$ to follow $z_4$ under the same conditions.
The phenomenon we are interested in is that where, in the above configuration,
\begin{equation}\label{cond15}
d(z_1',z_2')<d(z_1',z_3'),\, d(z_2',z_1')<d(z_2',z_3'),
\end{equation}
where $z_1'=\frac{z_1+z_3}{2}$, $z_2'=\frac{z_2+z_3}{2}$, and $z_3'=\frac{z_3+z_4}{2}$. That is $z_1'$ and $z_2'$ are mutual closest neighbors in the absence of other points. Note that $\mathcal{D}$ is a semi-algebraic set \cite{mohab}.

For the points $(z_1, ..., z_4)$ to be involved in the formation of a leader pair of type 1, it is necessary but not sufficient that $(z_1,\ldots,z_4)$ belong to ${\cal D}$. 
For this to happen, in addition, it must be that there are no other points of the Poisson P.P. that change the facts
that both $z_1$ and $z_2$ follow $z_3$, $z_3$ follows $z_4$, and $z_1'$ and $z_2'$ are mutually nearest neighbors. 

Let $\Phi_{1,2,3,4}=\delta_{z_1}+\delta_{z_2}+\delta_{z_3}+\delta_{z_4}$ denote the point process $\Phi$ restricted to the set $z_1,z_2,z_3$, and $z_4$. 
Let $N_1^{(4)}$ be the point process of quadruples of points of the Poisson P.P. $\Phi$ that belong to ${\cal D}$,
and are such that the following event $M_1$ holds: in $\Phi$, both $z_1$ and $z_2$ follow $z_3$ and $z_3$ follows $z_4$.

In a first step, we evaluate the spatial frequency $\beta^1$ of the event $M_1$, which is an upper bound on the frequency of leader pairs of order 1, type 1.

For $A\in \mathcal{B}^4$, the mean measure of $N_1^{(4)}$ is given by,
\begin{align*}
&\mathbb{E}[N_1^{(4)}(A )]  =  \mathbb{E}
[\sum_{z_1, z_2, z_3, z_4 \in \Phi}^{\neq}1_{z_1,z_2,z_3,z_4\in A \cap \mathcal{D}}
1_{\Phi(B(z_1 \rightarrow z_3))=1}
 1_{\Phi(B(z_2 \rightarrow z_3))=1 }1_{\Phi(B(z_3 \rightarrow z_4))=1 }],
\end{align*}
where $\Phi$ is the P.P.P. and $B(x\rightarrow y)$ is the open ball of center $x$ and
radius $|x-y|$.
In integral form, this is
\begin{align*}
&\mathbb{E}[N_1^{(4)}(A )]=\mathbb{E}[\int_{\mathbb{R}^2\times \mathbb{R}^2 \times \mathbb{R}^2 \times \mathbb{R}^2}1_{z_1,z_2,z_3,z_4\in A \cap \mathcal{D}}1_{\Phi(B(z_1 \rightarrow z_3))=1}1_{\Phi(B(z_2 \rightarrow z_3))=1 } \\ &
1_{\Phi(B(z_3 \rightarrow z_4))=1 } \Phi^{(4)}({\rm d}z_1\times {\rm d}z_2 \times{\rm d}z_3 \times {\rm d}z_4) ].
\end{align*}
By the higher order Campbell-Little-Mecke formula,
\begin{align*}
&\mathbb{E}[N_1^{(4)}(A )]=\int_{\mathbb{R}^2\times \mathbb{R}^2 \times \mathbb{R}^2 \times \mathbb{R}^2}1_{z_1,z_2,z_3,z_4\in A \cap \mathcal{D}}\\
& \mathbb{E}^{z_1,z_2,z_3,z_4}[1_{\Phi(B(z_1 \rightarrow z_3))=1}1_{\Phi(B(z_2 \rightarrow z_3))=1 } 1_{\Phi(B(z_3 \rightarrow z_4))=1 }]\lambda^{(4)}({\rm d}z_1{\rm d}z_2{\rm d}z_3{\rm d}z_4) ,
\end{align*}
where $\mathbb{E}^{z_1,z_2,z_3,z_4}$ is the Palm expectation of $\Phi$ at $z_1, ..., z_4$, and $\lambda^{(4)}$
is the Poisson factorial moment measure of order 4.
By Slivnyak's theorem,
\begin{equation*}
\begin{split}
\mathbb{E}[N^{(4)}_1(A )]=\int_{\mathbb{R}^2\times \mathbb{R}^2 \times \mathbb{R}^2 \times \mathbb{R}^2}1_{z_1,z_2,z_3,z_4\in A \cap \mathcal{D}}
\mathbb{E}[1_{\hat{\Phi}(B(z_1 \rightarrow z_3))=1}
1_{\hat{\Phi}(B(z_2 \rightarrow z_3))=1 } 1_{\hat{\Phi}(B(z_3 \rightarrow z_4))=1 }]
\lambda^4 {\rm d}z_1 {\rm d}z_2 {\rm d}z_3 {\rm d}z_4 ,
\end{split}
\end{equation*}
Note that $z_1\in B(z_1\to z_3)$, $z_2\in B(z_2\to z_3)$, and
$z_3\in B(z_3\to z_4)$ because all balls are open and centered at
their first argument. Consequently, under the Palm process
$\hat{\Phi}=\Phi+\delta_{z_1}+\delta_{z_2}+\delta_{z_3}+\delta_{z_4}$,
each of the above balls already contains exactly one deterministic Palm
point. Therefore the conditions
\[
\hat{\Phi}(B(z_1\to z_3))=1,\qquad
\hat{\Phi}(B(z_2\to z_3))=1,\qquad
\hat{\Phi}(B(z_3\to z_4))=1
\]
are equivalent to the statement that the underlying Poisson process
$\Phi$ contains no additional points in these balls. This yields the
usual Poisson void probability of the union of the three balls.
We can simplify the expression a bit more using stationarity of the Poisson point process. Take $A=\mathbb{R}^2\times \mathbb{R}^2\times C \times \mathbb{R}^2$ with $C\in \mathcal{B}(\mathbb{R}^2)$ a compact. Because we have a stationary point process,
using the change of variables,  $\tilde{z}_1=z_1-z_3$,
$\tilde{z}_2=z_2-z_3$, and $\tilde{z}_4=z_4-z_3$, we get
\begin{eqnarray}
\label{combined}
\mathbb{E}[N_1^{(4)}(\mathbb{R}^2\times \mathbb{R}^2\times C \times \mathbb{R}^2)]=
|C| \int_{{\mathbb{R}^2}}\int_{{\mathbb{R}^2}} \int_{{\mathbb{R}^2}} 
1_{\tilde{z}_1,\tilde{z}_2,\tilde{z}_4 \in \tilde{\cal D}} \nonumber
e^{-\lambda (\mathrm{Vol}(B(0\rightarrow \tilde{z}_4)\cup B(\tilde{z}_1\rightarrow 0)\cup B(\tilde{z}_2 \rightarrow 0))}
\nonumber
\lambda^4 {\rm d}\tilde{z}_1{\rm d}\tilde{z}_2{\rm d}\tilde{z}_4=|C|\beta_1,
\end{eqnarray}
with $\tilde{{\cal D}}=\{(z_1-z_3,\, z_2-z_3,\, 0,\, z_4-z_3): (z_1,\, z_2,\, z_3,\, z_4)\in {\cal D}\}$.

This is the announced integral over a semi-algebraic set. Note that this integral is twice the density of interest since we do not distinguish whether $d(z_1,z_3)< d(z_2,z_3)$ or the other way around. In order to evaluate the integral, we need to write a formula for the volume of the union of 3 balls. This is discussed in Section \ref{semi_algebraic_numerics}.

Now that we have all the elements, we have a formula to calculate exact value of the volume $|\mathbb{U}(z_1,z_2,z_3)|$ of the union of balls. For compactness, the integral of the frequency of type 1 can be written as:
\begin{equation}\label{combined}
\mathbb{E}[N_1^{(4)}(\mathbb{R}^2\times \mathbb{R}^2\times C \times \mathbb{R}^2)]=
|C| \int_{{\mathbb{R}^2}}\int_{{\mathbb{R}^2}} \int_{{\mathbb{R}^2}} 
\lambda^41_{\tilde{z}_1,\tilde{z}_2,\tilde{z}_4 \in {\cal D}_0} \nonumber
e^{-\lambda |\mathbb{U}(0, \tilde{z}_1,\tilde{z}_2,\tilde{z}_4)|}
\nonumber
 {\rm d}\tilde{z}_1{\rm d}\tilde{z}_2{\rm d}\tilde{z}_4:=|C|\beta^1,
\end{equation}
with ${\cal D}_0 ={\cal D} \cap \{z_3=0\}$.



%
%
%
We now discuss the exact calculation. In order to get the event of interest, one should in addition have $\Phi_1(B(z_1\to z_2))=1$
and $\Phi_1(B(z_2\to z_1))=1$, with $\Phi_1$ the point process at step 1. 
In other words, certain refinements of the last configuration should be removed from the counting.
We can order these refinements in the disjoint and exhaustive categories listed below
\begin{enumerate}
\item There is an extra point $x$ in $\Phi$ that follows $z_1$ in $\Phi$ and such that the distance from
$x'$ to $z_1'$ is less than that from $z_1'$ to $z'_2$. The case where there is an extra point $x$ in $\Phi$ that follows $z_2$ in $\Phi$ and such that the distance $d(x,z_2, \Phi_1)< d(z_1,z_2,\Phi_1)$ is analogous and is counted here as well. This is due to a fact that we do not distinguish whether $d(z_1,z_3)<d(z_2,z_3)$ or vice versa.

For evaluating the frequency of this event, we have to enrich ${\cal D}$ by adding $x$ and state that $x$ follows $z_1$ and that none of the points $z_1,z_2,z_3,z_4$ follows $x$, that is
\begin{equation}
\begin{split}
d(x,z_1)<d(x,z_2),\, d(x,z_1)<d(x,z_3), \,d(x,z_1)<d(x,z_4), \\
d(z_1,z_3)<d(z_1,x), \,d(z_2,z_3)<d(z_2,x),\, d(z_4,z_3)<d(z_4,x)
,d(z_3,z_4)<d(z_3,x).
\end{split}
\end{equation}

This adds 7 quadratic inequalities. 
Finally, the condition that $d(x,z_1, \Phi_1)< d(z_1,z_2,\Phi_1)$, gives one more quadratic inequality 
\begin{equation}
d(x',z_1')<d(z'_1,z'_2),
\end{equation}
where $x'=\frac{x+z_1}{2}$.
So the frequency of this refinement can also be reduced to the evaluation of an integral
over a semi-algebraic set ${\cal D}_1$, which is a refinement of ${\cal D}$ with one more variable
and 8 more quadratic inequalities, with the function to be integrated involving one more ball in the union.

For $A\in \mathcal{B}^5$, the mean measure of $N_{1,1}^{(5)}$, which is the point process of the 5-tuples of points satisfying the above conditions, is given by,
\begin{eqnarray*}
\mathbb{E}[N_{1,1}^{(5)}(A )] & = & \mathbb{E}
[\sum_{x,z_1, z_2, z_3, z_4 \in \Phi}^{\neq}1_{x,z_1,z_2,z_3,z_4\in A \cap {\cal D}_1}
1_{\Phi(B(z_1 \rightarrow z_3))=1}\\&& 1_{\Phi(B(z_2 \rightarrow z_3))=1 }
1_{\Phi(B(z_3 \rightarrow z_4))=1 }1_{\Phi(B(x \rightarrow z_1))=1}],
\end{eqnarray*}

In integral form, this is
\begin{equation*}
\begin{split}
\mathbb{E}[N_{1,1}^{(5)}(A )]=\mathbb{E}[\int_{\mathbb{R}^2\times \mathbb{R}^2 \times \mathbb{R}^2 \times \mathbb{R}^2\times \mathbb{R}^2}1_{x,z_1,z_2,z_3,z_4\in A \cap {\cal D}_1}1_{\Phi(B(z_1 \rightarrow z_3))=1}\\1_{\Phi(B(z_2 \rightarrow z_3))=1 } 1_{\Phi(B(z_3 \rightarrow z_4))=1 } 1_{\Phi(B(x \rightarrow z_1))=1}\Phi^{(5)}({\rm d}x\times {\rm d}z_1\times {\rm d}z_2 \times {\rm d}z_3 \times {\rm d}z_4) ].
\end{split}
\end{equation*}
By the higher order Campbell-Little-Mecke formula,

\begin{align*}
\mathbb{E}[N_{1,1}^{(5)}(A )] &=\int_{\mathbb{R}^2\times \mathbb{R}^2 \times \mathbb{R}^2 \times \mathbb{R}^2\times \mathbb{R}^2}1_{x, z_1,z_2,z_3,z_4\in A \cap {\cal D}_1}\\
& \mathbb{E}^{x,z_1,z_2,z_3,z_4}[1_{\Phi(B(z_1 \rightarrow z_3))=1}1_{\Phi(B(z_2 \rightarrow z_3))=1 } 1_{\Phi(B(z_3 \rightarrow z_4))=1 }
1_{\Phi(B(x \rightarrow z_1))=1}]\lambda^{(5)}({\rm d}x{\rm d}z_1{\rm d}z_2{\rm d}z_3{\rm d}z_4) ,
\end{align*}

where $\mathbb{E}^{x, z_1,z_2,z_3,z_4}$ is the Palm expectation of $\Phi$ at $x,z_1,...z_4$, and $\lambda^{(5)}$
is the factorial Poisson moment measure of order 5.
By Slivnyak's theorem,
\begin{equation*}
\begin{split}
\mathbb{E}[N_{1,1}^{(5)}(A )]=\int_{\mathbb{R}^2\times \mathbb{R}^2 \times \mathbb{R}^2 \times \mathbb{R}^2\times \mathbb{R}^2}1_{x,z_1,z_2,z_3,z_4\in A \cap {\cal D}_1}
\mathbb{E}[1_{\hat{\Phi}(B(z_1 \rightarrow z_3))=1}\\
1_{\hat{\Phi}(B(z_2 \rightarrow z_3))=1 } 1_{\hat{\Phi}(B(z_3 \rightarrow z_4))=1 }1_{\hat{\Phi}(B(x \rightarrow z_1))=1}]
\lambda^5 {\rm d}x {\rm d}z_1 {\rm d}z_2 {\rm d}z_3 {\rm d}z_4 ,
\end{split}
\end{equation*}
where $\hat{\Phi}=\Phi+\delta_x+\delta_{z_1}+\delta_{z_2}+\delta_{z_3}+\delta_{z_4}$ is a Poisson P.P with intensity $\lambda$.
Take $A=\mathbb{R}^2\times \mathbb{R}^2\times \mathbb{R}^2\times C \times \mathbb{R}^2$ with $C$ a compact. 
Because we have a stationary point process, 
using the change of variables $\tilde{z}_1=z_1-z_3$, $\tilde{x}=x-z_3$,
$\tilde{z}_2=z_2-z_3$, and $\tilde{z}_4=z_4-z_3$, we get
\begin{eqnarray}
\label{total_type1_case1}
\mathbb{E}[N_{1,1}^{(5)}(\mathbb{R}^2\times \mathbb{R}^2\times \mathbb{R}^2\times C \times \mathbb{R}^2)]=
|C| \int_{{\mathbb{R}^2}}\int_{{\mathbb{R}^2}}\int_{{\mathbb{R}^2}} \int_{{\mathbb{R}^2}} 
1_{\tilde{x},\tilde{z}_1,\tilde{z}_2,\tilde{z}_4 \in \tilde{{\cal D}_1}} \nonumber
\\
e^{-\lambda (\mathrm{Vol}(B(0\rightarrow \tilde{z}_4)\cup B(\tilde{z}_1\rightarrow 0)\cup B(\tilde{z}_2 \rightarrow 0)\cup B(\tilde{x} \rightarrow \tilde{z}_1)))} 
\nonumber
\lambda^5 {\rm d}\tilde{x}{\rm d}\tilde{z}_1{\rm d}\tilde{z}_2{\rm d}\tilde{z}_4,
\end{eqnarray}
with $\tilde{{\cal D}_1} =\{(x-z_3,\, z_1-z_3,\, z_2-z_3,\, 0,\, z_4-z_3): (x, \, z_1,\, z_2,\, z_3,\, z_4)\in {\cal D}_1\}$.

Similar to the derivation for the union of 3 balls, one can calculate the volume of the union of 4 balls. 




\item There is an extra point $x$ in $\Phi$ that follows $z_3$ in $\Phi$ and such that $d(x,z_1, \Phi_1)< d(z_1,z_2,\Phi_1)$. The case where $d(x,z_2, \Phi_1)< d(z_1,z_2,\Phi_1)$ is symmetric.

For evaluating the frequency of this event, we have to enrich $\cal D$ by adding $x$ and state that $x$ follows $z_3$ and that none of the points $z_1,z_2,z_3, z_4$ follow $x$.
\begin{equation}
\begin{split}
d(x,z_3)<d(x,z_1),\, d(x,z_3)<d(x,z_2),\, d(x,z_3)<d(x,z_4), \\
d(z_1,z_3)<d(z_1,x), \,d(z_2,z_3)<d(z_2,x),\, d(z_4,z_3)<d(z_4,x),  d(z_3,z_4)<d(z_3,x).
\end{split}
\end{equation}

This adds 7 quadratic inequalities. 
Finally, the condition that $d(x,z_1, \Phi_1)< d(z_1,z_2,\Phi_1)$, gives one more quadratic inequality:
\begin{equation}
d(x',z_1')<d(z'_1,z'_2),
\end{equation}
where $x'=\frac{x+z_3}{2}$.
So the frequency of this refinement can be reduced to the evaluation of an integral
over a semi-algebraic set ${\cal D}_2$, which is a refinement of ${\cal D}$ with one more variable
and 8 more quadratic inequalities, with the function to be integrated involving one more ball in the union.

For $A\in \mathcal{B}^5$, the mean measure of $N_{1,2}^{(5)}$, which is the point process of the 5-tuples of points satisfying the conditions above, is given by,
\begin{eqnarray*}
\mathbb{E}[N_{1,2}^{(5)}(A )] & = & \mathbb{E}
[\sum_{x,z_1, z_2, z_3, z_4 \in \Phi}^{\neq}1_{x,z_1,z_2,z_3,z_4\in A \cap {\cal D}_2}
1_{\Phi(B(z_1 \rightarrow z_3))=1}\\&& 1_{\Phi(B(z_2 \rightarrow z_3))=1 }
1_{\Phi(B(z_3 \rightarrow z_4))=1 }1_{\Phi(B(x \rightarrow z_3))=1}].
\end{eqnarray*}

In integral form, this is
\begin{equation*}
\begin{split}
\mathbb{E}[N_{1,2}^{(5)}(A )]&=\mathbb{E}[\int_{\mathbb{R}^2\times \mathbb{R}^2 \times \mathbb{R}^2 \times \mathbb{R}^2\times \mathbb{R}^2}1_{x,z_1,z_2,z_3,z_4\in A \cap {\cal D}_2}1_{\Phi(B(z_1 \rightarrow z_3))=1}, \\& 1_{\Phi(B(z_2 \rightarrow z_3))=1 } 1_{\Phi(B(z_3 \rightarrow z_4))=1 } 1_{\Phi(B(x \rightarrow z_3))=1}\Phi^{(5)}({\rm d}x\times {\rm d}z_1\times {\rm d}z_2 \times {\rm d}z_3 \times {\rm d}z_4) ].
\end{split}
\end{equation*}
By the higher order Campbell-Little-Mecke formula, 
\begin{equation*}
\begin{split}
\mathbb{E}[N_{1,2}^{(5)}(A )]&=\int_{\mathbb{R}^2\times \mathbb{R}^2 \times \mathbb{R}^2 \times \mathbb{R}^2\times \mathbb{R}^2}1_{x, z_1,z_2,z_3,z_4\in A \cap {\cal D}_2} \mathbb{E}^{x,z_1,z_2,z_3,z_4}[1_{\Phi(B(z_1 \rightarrow z_3))=1}1_{\Phi(B(z_2 \rightarrow z_3))=1 }
\\ & 1_{\Phi(B(z_3 \rightarrow z_4))=1 }
1_{\Phi(B(x \rightarrow z_3))=1}]\lambda^{(5)}({\rm d}x{\rm d}z_1{\rm d}z_2{\rm d}z_3{\rm d}z_4) ,
\end{split}
\end{equation*}
where $\mathbb{E}^{x, z_1,z_2,z_3,z_4}$ is the Palm expectation of $\Phi$ at $x,z_1,...z_4$, and $\lambda^{(5)}$
is the factorial Poisson moment measure of order 5.
Following the analogous steps as in case one, stationarity and the change of variables $\tilde{z}_1=z_1-z_3$, $\tilde{x}=x-z_3$,
$\tilde{z}_2=z_2-z_3$, and $\tilde{z}_4=z_4-z_3$, we get
\begin{eqnarray}
\label{total_type1_case2}
\mathbb{E}[N_{1,2}^{(5)}(\mathbb{R}^2\times \mathbb{R}^2\times \mathbb{R}^2\times C \times \mathbb{R}^2)]=
|C| \int_{{\mathbb{R}^2}}\int_{{\mathbb{R}^2}}\int_{{\mathbb{R}^2}} \int_{{\mathbb{R}^2}} 
1_{\tilde{x},\tilde{z}_1,\tilde{z}_2,\tilde{z}_4 \in \tilde{{\cal D}_2}} \nonumber
\\
e^{-\lambda (\mathrm{Vol}(B(0\rightarrow \tilde{z}_4)\cup B(\tilde{z}_1\rightarrow 0)\cup B(\tilde{z}_2 \rightarrow 0)\cup B(\tilde{x} \rightarrow 0)))} 
\nonumber
\lambda^5 {\rm d}\tilde{x}{\rm d}\tilde{z}_1{\rm d}\tilde{z}_2{\rm d}\tilde{z}_4,
\end{eqnarray}
with $\tilde{{\cal D}_2} =\{(x-z_3,\, z_1-z_3,\, z_2-z_3,\, 0,\, z_4-z_3): (x, \, z_1,\, z_2,\, z_3,\, z_4)\in {\cal D}_2\}$.

Again, one can calculate the volume of the union of 4 balls.

\item There are two extra points $x$ and $y$, with $x\neq y$, such that $x$ follows $y$ (4 inequalities) and none of the points
$z_1,\ldots,z_4$ follow either $x$ or $y$ (6 inequalities). Thus we have
\begin{gather*}
d(x,y)<d(x,z_1), \,d(x,y)<d(x,z_2), \, d(x,y)< d(x,z_3), \, d(x,y)<d(x,z_4),\\
d(z_1,z_3)<d(z_1,x), \,d(z_1,z_3)<d(z_1,y),\\
d(z_2,z_3)<d(z_2,x),\, d(z_2,z_3)<d(z_2,y),\\
d(z_3,z_4)<d(z_3,x),\, d(z_3,z_4)<d(z_3,y).
\end{gather*}

To this, as above one should add the condition that $x'$ is closer from $z_1'$ than $z_1'$ to $z_2'$. 
WLOG take $x'=\frac{x+y}{2}$ to be closer to $z'_1$ than $z'_1$ to $z'_2$, i.e.,
\[d(z'_1,x')<d(z'_1,z'_2).\]
This amounts to 2 more variables and 11 more quadratic inequalities 
and one more empty ball conditions ($\Phi(B(x\to y))=1$).
So the frequency of this refinement can be reduced to the evaluation of an integral
over a semi-algebraic set ${\cal D}_3$, which is a refinement of ${\cal D}$ with two more variables
and 11 more quadratic inequalities, with the function to be integrated involving one more ball in the union.

For $A\in \mathcal{B}^6$, the mean measure of $N_{1,3}^{(6)}$, which is the point process of the 6-tuples of points satisfying the conditions above, is given by
\begin{eqnarray*}
\mathbb{E}[N_{1,3}^{(6)}(A )] & = & \mathbb{E}
[\sum_{x,y, z_1, z_2, z_3, z_4 \in \Phi}^{\neq}1_{x,y, z_1,z_2,z_3,z_4\in A \cap {\cal D}_3}
1_{\Phi(B(z_1 \rightarrow z_3))=1} \\&& 1_{\Phi(B(z_2 \rightarrow z_3))=1 }
1_{\Phi(B(z_3 \rightarrow z_4))=1 }1_{\Phi(B(x \rightarrow y))=1}],
\end{eqnarray*}

In integral form, this is
\begin{equation*}
\begin{split}
\mathbb{E}[N_{1,3}^{(6)}(A )]&=\mathbb{E}[\int_{\mathbb{R}^2\times \mathbb{R}^2\times \mathbb{R}^2 \times \mathbb{R}^2 \times \mathbb{R}^2\times \mathbb{R}^2}1_{x, y, z_1,z_2,z_3,z_4\in A \cap {\cal D}_3} 1_{\Phi(B(z_1 \rightarrow z_3))=1}, 1_{\Phi(B(z_2 \rightarrow z_3))=1 } 1_{\Phi(B(z_3 \rightarrow z_4))=1 }\\ & 1_{\Phi(B(x \rightarrow y))=1}\Phi^{(6)}({\rm d}x \times {\rm d}y \times {\rm d}z_1\times {\rm d}z_2 \times {\rm d}z_3 \times {\rm d}z_4) ].
\end{split}
\end{equation*}
Analogously to case one, the calculation boils down to having $A=\mathbb{R}^2\times \mathbb{R}^2\times \mathbb{R}^2\times \mathbb{R}^2\times C \times \mathbb{R}^2$ with $C$ a compact. 
Because we have a stationary point process, 
using the change of variables $\tilde{z}_1=z_1-z_3$, $\tilde{x}=x-z_3$, $\tilde{y}=y-z_3$
$\tilde{z}_2=z_2-z_3$, and $\tilde{z}_4=z_4-z_3$, we get
\begin{align*}
\label{total_type1_case3}
\mathbb{E}[N_{1,3}^{(6)}(\mathbb{R}^2\times \mathbb{R}^2\times \mathbb{R}^2\times \mathbb{R}^2\times C \times \mathbb{R}^2)]&=
|C| \int_{{\mathbb{R}^2}}\int_{{\mathbb{R}^2}}\int_{{\mathbb{R}^2}} \int_{{\mathbb{R}^2}} 
1_{\tilde{x},\tilde{z}_1,\tilde{z}_2,\tilde{z}_4 \in \tilde{{\cal D}_3}} \nonumber
\\&
e^{-\lambda (\mathrm{Vol}(B(0\rightarrow \tilde{z}_4)\cup B(\tilde{z}_1\rightarrow 0)\cup B(\tilde{z}_2 \rightarrow 0)\cup B(\tilde{x} \rightarrow \tilde{y})))} 
\nonumber
\lambda^6 d\tilde{x} {\rm d}\tilde{y} {\rm d}\tilde{z}_1{\rm d}\tilde{z}_2{\rm d}\tilde{z}_4,
\end{align*}
with 

\[\tilde{{\cal D}_3} =\{(x-z_3,\, y-z_3,\, z_1-z_3,\, z_2-z_3,0, z_4-z_3): (x, y, z_1,\, z_2,\, z_3,\, z_4)\in {\cal D}_3\}.\]

\end{enumerate}

Let $\beta_i$ the frequency of event $i=1,2,3$ in the last list.
Then the frequency of interest is
\begin{equation}\label{total_frequency_order1type1}
\frac{1}{2}(\beta^1-\sum_{i=1}^3 \beta_i).
\end{equation}

Note that the coefficient $\frac{1}{2}$ comes from a symmetry argument and to avoid double counting.
We use numerical methods to evaluate all the frequencies. This is discussed in the next section.

\section{Numerical Estimation Methods}\label{numerical_methods}

We use two different methods to estimate frequencies. The first is \textit{discrete event spatial simulation}, which consist in sampling a Poisson point process and exploiting the ergodic properties of the Poisson point process and related factors of such point processes \cite{kendall_mecke, baccelli_nova_knjiga}. This approach is grounded in the ergodic theory of point processes \cite{kendall_mecke}.
The second method involves the numerical evaluation of integrals based on semi-algebraic sets discussed above. We explain both methods in detail and provide estimated values for some of the frequencies mentioned earlier. For each, we first provide some theoretical background and then describe how we apply it for our evaluation. When appropriate, we include 95\% confidence intervals. Finally, we compare the results obtained from both methods.

\subsection{Discrete Event Spatial Simulation}\label{spatial_simulation}


The background about ergodicity discussed in this section comes from \cite{baccelli_nova_knjiga}.

\begin{definition}\cite{baccelli_nova_knjiga}
A stationary framework $(\Omega, \mathcal{A},\{\theta_t\}_{t\in \mathbb{R}^d},\mathbb{P})$ is said to be ergodic if
\[\lim_{a\rightarrow \infty} \frac{1}{(2a)^d}\int_{[-a,a]^d}\mathbb{P}(A_1 \cap \theta_{-x}A_2){\rm d}x=\mathbb{P}(A_1)\mathbb{P}(A_2), \, \, \forall A_1, A_2, \in \mathcal{A}.\]
It is said to be mixing if 
\[\lim_{|x|\rightarrow \infty} \mathbb{P}(A_1\cap \theta_{-x}A_2)=\mathbb{P}(A_1)\mathbb{P}(A_2), \, \, \forall A_1, A_2 \in \mathcal{A}.\]
\end{definition}

\begin{lemma}
For all $n$, the follower point process of order $n$ is a factor of a Poisson point process, and is hence ergodic and mixing.
\end{lemma}

\textit{Discrete event spatial simulation} (DESS) starts with sampling a Poisson point process in the fixed bounded  window $W \subset \mathbb{R}^2$ (the observation window). We pick a convex set $W$, typically a square. In order to simulate a homogeneous Poisson point process, we use the ideas from \cite{kendall_mecke}. In \cite{kendall_mecke} it is suggested that simulating in a compact region $W$ can be split into two parts. Namely, the number of points in $W$ is determined from the simulation of a Poisson random variable. Then the positions of points in $W$ are obtained from simulating a binomial point process with that number of points. Since we are interested in the point process in the whole of $\mathbb{R}^2$, we take $W$ big. Hence there are small boundary effects. 

 Then we iteratively apply the Follower Dynamics to all the points simultaneously. We still have the mixing property satisfied. The first set of simulations is done using MATLAB and applying the Monte Carlo method. For each agent, we keep track of its leader at every step. In order to calculate any frequency, we count the number of agents in the window for which our sets of conditions are satisfied, and we divide it by the total number of agents in the window. 

For example, for the frequency of leader pairs of order 0, we check for all pairs of agents and count the number of mutually closest neighbors at step 0. Since we keep track of the leader for each agent, we just check whether two agents are each others' leaders. 

Similarly, in order to determine the frequency of leader pairs of order 1, type 1, we check, for all pairs of agents, whether they had the same leader (but were not each others' leaders) at step 0 and are each others' leaders at step 1.

For the frequency of leader pairs of order 1, type 2, we check, for all pairs of agents, whether one was a leader of the other at step 0 and both are each others' leaders at step 1.

Table~\ref{table:frequencies} contains frequencies and $95\%$ confidence interval obtained using discrete event spatial simulation. In order to obtain a $95\%$ confidence interval we run simulation many times and keep track of the statistics. The statistics is obtained from $40$ samples 
containing in mean 20000 points each. For instance, observe that the frequency of leader pairs of order 1 is in the interval $[0.021, 0.023]$ with $95\%$ confidence, which is the sum of both types shown in Table~\ref{table:frequencies}. Similarly, the frequency of inversion is in the interval $[0.13, 0.14]$ with the same confidence. 


\begin{table}
\begin{center}
 \begin{tabular}{||c c  ||} 
 \hline
 Frequency of &  $95\%$ Confidence interval  \\ [0.5ex] 
 \hline\hline
 Leader pair of order 0 & $[0.6203, 0.6227]$  \\  [0.5ex]
\hline
Order 1 type 1 & $[0.005, 0.006]$ \\  [0.5ex]
\hline
Order 1 type 2 & $[0.0100, 0.0105]$  \\  [0.5ex]
\hline
4 body swap & $[0.000063, 0.0001]$  \\  [0.5ex]
\hline
Inversion Type 1 & $[0.138, 0.140]$  \\  [0.5ex]
\hline
Inversion Type 2 & $[0.00019, 0.00027]$  \\  [0.5ex]
\hline
\end{tabular} 
\caption{\label{table:frequencies} Frequencies and $95\%$ confidence intervals obtained using DESS}
 \end{center}  
\end{table}


\subsection{Integrals on Semi-Algebraic Sets}\label{semi_algebraic_numerics}

\subsubsection{Union of Balls}

When exploiting our integral geometry representations of frequencies in the examples explored in Section \ref{sc:frequency_phenomena}, we end up with expressions that are integrals involving unions of balls. Hence, this way to find the probabilities of interest requires to compute the volumes of unions of balls. In this subsection, we first represent these unions of balls as semi-algebraic sets.

The setting is $\mathbb{R}^d$. Let $k$ be a positive integer.
Let $c_1,\ldots,c_k$ be arbitrary points of $\mathbb{R}^d$ which are the centers
of balls of positive radii $r_1,\ldots,r_k$, respectively.
We are interested in the volume of
$$ \mathbb{U}:=\cup_{i=1}^k B(c_i,r_i),$$
with $B(c,r)$ the closed ball of center $c$ and radius $r$.

Let $\mathcal S$ be set of all non-empty subsets of $[1,\ldots,k]$.  
For all $s\in S$, let 
$$\mathbb{V}_s= \cap_{i\in s}^k B(c_i,r_i)\cap_{j\notin s}^k B(c_j,r_j)^c.$$
We have $\mathbb{U}= \cup_{s\in S} \mathbb{V}_s,$
where the last union is disjoint. 

Hence
$$\mid\mathbb{U} \mid= \sum_{s\in S} \mid \mathbb{V}_s \mid .$$
But for all $s$, $\mathbb{V}_s$ is the following semi-algebraic set:
\begin{eqnarray*}
|| z-x_i || & \le & r_i,\quad i\in s,\\
|| z-x_j || & > & r_j,\quad i\notin s.
\end{eqnarray*}
Hence $\mathbb{U}$ is the disjoint union of the $2^k-1$ semi-algebraic sets
$\{\mathbb{V}_s,s\in S\}$ and its volume is the sum of the volumes of these
semi-algebraic sets of $\mathbb R^d$.




\subsubsection{Numerical evaluation of $\beta^1$}
\label{sec:numerical_beta1}

In this subsection, we derive the integral representation for the quantity
$\beta^1$ using the semi-algebraic characterization of the
underlying geometric constraints.
The resulting six-dimensional integral is then evaluated numerically
by Monte Carlo quadrature.

\subsubsection*{Semi-algebraic characterization of the integration domain}

By stationarity, we fix $x_3=0$, obtaining
\begin{equation}
  \beta^1
  =
  \lambda^4
  \int_{\mathbb R^2\times\mathbb R^2\times\mathbb R^2}
  \mathbf 1_{(\tilde x_1,\tilde x_2,\tilde x_4)\in\mathcal D_0}
  \,
  e^{-\lambda |\mathbb U(\tilde x_1,\tilde x_2,\tilde x_4)|}
  \,
  {\rm d}\tilde x_1\,{\rm d}\tilde x_2\,{\rm d}\tilde x_4,
  \label{eq:beta1_integral}
\end{equation}
where $\mathcal D_0=\mathcal D\cap\{x_3=0\}.$ Writing
$
r_i=|\tilde x_i|,
d_{ij}=|\tilde x_i-\tilde x_j|,
$
the defining conditions become
\begin{equation}
  r_1 < d_{12},
  \qquad
  r_1 < d_{14},
  \qquad
  r_2 < d_{12},
  \qquad
  r_2 < d_{24},
\end{equation}
\begin{equation}
  r_4 < r_1,
  \qquad
  r_4 < r_2,
  \qquad
  d_{12} < d_{14},
  \qquad
  d_{12} < d_{24}.
  \label{eq:D0_conditions}
\end{equation}

The volume of the geometric exclusion region is
\begin{equation}
  |\mathbb U(\tilde x_1,\tilde x_2,\tilde x_4)|
  =
  \mathrm{Vol }\!\Big(
  B(0,r_4)
  \cup
  B(\tilde x_1,r_1)
  \cup
  B(\tilde x_2,r_2)
  \Big),
  \label{eq:U_def}
\end{equation}
namely the area of the union of the three disks associated with the
nearest-neighbor constraints.

Writing
\[
D_0=B(0,r_4),
\qquad
D_1=B(\tilde x_1,r_1),
\qquad
D_2=B(\tilde x_2,r_2),
\]
the inclusion--exclusion formula gives
\begin{align}
|\mathbb U|
&=
|D_0|
+
|D_1|
+
|D_2|
-
|D_0\cap D_1|
-
|D_0\cap D_2|
-
|D_1\cap D_2|
+
|D_0\cap D_1\cap D_2|.
\label{eq:IE}
\end{align}

After translating $x_3$ to the origin,
$d(0,\tilde x_1)=r_1, d(0,\tilde x_2)=r_2,$
so all pairwise overlaps depend only on
$r_1,r_2,r_4,d_{12}$.
Pairwise intersection areas are evaluated using the exact lens formula
for disk intersections \cite{Fewell2006}.

The integration domain $\mathcal D_0$ also admits an exact
semi-algebraic description.
Indeed, after squaring the inequalities in
\eqref{eq:D0_conditions}, each constraint becomes polynomial in the
coordinates of
$(\tilde x_1,\tilde x_2,\tilde x_4)\in\mathbb R^6$.
Thus $\mathcal D_0$ is a semi-algebraic subset of $\mathbb R^6$ in
the sense of real algebraic geometry.

The numerical procedure therefore separates naturally into two parts:

\begin{enumerate}
    \item an exact geometric characterization of the region
    satisfying the nearest-neighbor constraints through polynomial
    inequalities and exact overlap formulas;

    \item a numerical approximation of the resulting integral by Monte
    Carlo quadrature.
\end{enumerate}

Note that Monte Carlo sampling is used only for numerical
integration and not for identifying the region itself.
Belonging to $\mathcal D_0$ is checked exactly for each sampled
configuration by evaluating the inequalities in
\eqref{eq:D0_conditions}.

\subsubsection{Monte Carlo evaluation of the integral}

The integral~\eqref{eq:beta1_integral} is six-dimensional and involves
nonlinear geometric constraints.
In this setting Monte Carlo quadrature is considerably more effective than
deterministic quadrature schemes.

The integration domain is described explicitly by the polynomial
inequalities~\eqref{eq:D0_conditions}, so the semi-algebraic structure
is used to determine exactly whether a sampled configuration belongs to
$\mathcal D_0$.
The Monte Carlo procedure is used only to approximate the resulting
integral numerically.

We first rewrite the integral in polar coordinates.
For each point
\[
\tilde x_i=(x_i,y_i)\in\mathbb R^2,
\]
we introduce the change of variables
\[
\tilde x_i
=
(r_i\cos\theta_i,r_i\sin\theta_i),
\qquad
r_i\ge0,
\quad
\theta_i\in[0,2\pi).
\]
The Jacobian determinant of this transformation is $r_i$, so
\[
{\rm d}\tilde x_i
=
r_i\,{\rm d}r_i\,{\rm d}\theta_i.
\]

The transformed integral is evaluated by Monte Carlo quadrature over
the truncated parameter domain
\[
\mathcal Q
=
[0,R_{\max}]^3\times[0,2\pi)^3,
\]
with $R_{\max}=7$.
The truncation error decays exponentially because
\[
|\mathbb U|
\ge
\pi\max(r_1^2,r_2^2,r_4^2),
\]
so the integrand decays at least as fast as
\[
e^{-\lambda\pi r^2}.
\]

Sample configurations
\[
\xi^{(k)}
=
(\tilde x_1^{(k)},\tilde x_2^{(k)},\tilde x_4^{(k)})
\]
are generated independently in polar coordinates.
For each point, the angular variable
$\theta_i$
is sampled uniformly on
$[0,2\pi)$.
The radial variables are sampled from the density
\[
f_R(r)
=
\frac{2r}{R_{\max}^2},
\qquad
0\le r\le R_{\max},
\]
which is exactly the radial distribution induced by uniform area
measure on the disk
$B(0,R_{\max})$
expressed in polar coordinates.
This is exactly the radial distribution induced by uniform sampling on
the disk $B(0,R_{\max})$ expressed in polar coordinates.
For each sampled configuration satisfying
\eqref{eq:D0_conditions},
we compute
\[
e^{-\lambda |\mathbb U(\xi^{(k)})|}.
\]

The estimator is
\[
\hat\beta^1
=
\lambda^4
\frac1A
\sum_{k=1}^{A}
\mathbf 1_{\xi^{(k)}\in\mathcal D_0}
\,w(\xi^{(k)})
\,e^{-\lambda |\mathbb U(\xi^{(k)})|},
\]
where
$A=30000$
denotes the number of sampled configurations in each batch and
$w(\xi^{(k)})$
is the corresponding density-correction factor arising from the polar
sampling distribution.

To estimate variability, we perform
$n=10$
independent batches.

The procedure is rigorous because
\eqref{eq:beta1_integral}
is an ordinary Lebesgue integral over $\mathbb R^6$ with explicitly
specified indicator constraints.
The estimator above is the standard Monte Carlo approximation of this
integral by empirical averaging.
Since
\[
0
\le
\mathbf 1_{\xi\in\mathcal D_0}
e^{-\lambda|\mathbb U(\xi)|}
\le
1,
\]
the integrand is bounded and integrable.
Consequently, the strong law of large numbers implies that
\[
\hat\beta^1
\to
\beta^1
\qquad
\text{almost surely as }
A\to\infty.
\]

\subsubsection{Comparison}

As an independent validation of the integral computation, we estimate
the same quantity using the discrete event spatial simulation (DESS)
described in Section~5.1.

\paragraph{Numerical results.}

Table~\ref{tab:beta1}
reports the resulting estimates for
$\lambda=1$.
The close agreement between the integral estimator and the DESS provides a consistency check for both the
geometric characterization of the configuration space and the numerical
implementation of the two procedures.

\begin{table}[ht!]
\centering
\caption{
Comparison of two independent estimators of $\beta^1$ for $\lambda=1$.
}
\label{tab:beta1}
\begin{tabular}{lcc}
\toprule
Method
&
Estimate
&
95\% confidence interval
\\
\midrule

Integral
&
\texttt{0.00036}
&
\texttt{[0.00031,\ 0.00041]}
\\

DESS
&
\texttt{0.00035}
&
\texttt{[0.00031,\ 0.00040]}
\\

\bottomrule
\end{tabular}
\end{table}
\section{Asymptotic Regime}\label{sc:asymptotic}

This section focuses on the limiting behavior of parties. We analyze \textit{stable parties} in Section \ref{stable_chain}. Stable parties are special types of Parties, which are trees for which the graph structure doesn't change with steps and we defined them in Section 3.2. Such parties have a special structure. As we shall see they have no branching outside the root.  We conjecture that all the Follower Parties eventually become stable parties.

Finally, in Section \ref{long_term_relations}, we present some general long-term relations. As the number of steps in the Follower Dynamics approaches infinity, two sub-processes of the initial point process emerge. One converges in total variation - it is the ultimate leader pair point process, while the other converges weakly to its limit, the ultimate follower point process. Simulations in two dimensions show that the density of ultimate leader pairs is approximately 0.66. The remaining points, the ultimate followers, converge weakly to one of the ultimate leader pairs.

\subsection{Stable Party}\label{stable_chain}

 In this Section we focus on stable parties like in Figure~\ref{fig:whole_chain_7}, and show that for such parties we can derive exact positions of agents at any step. A simple example of a stable party is one leader pair and one follower.

\begin{figure}[h!]
\centering
\begin{minipage}{.5\textwidth}
  \centering
  \includegraphics[width=1\linewidth]{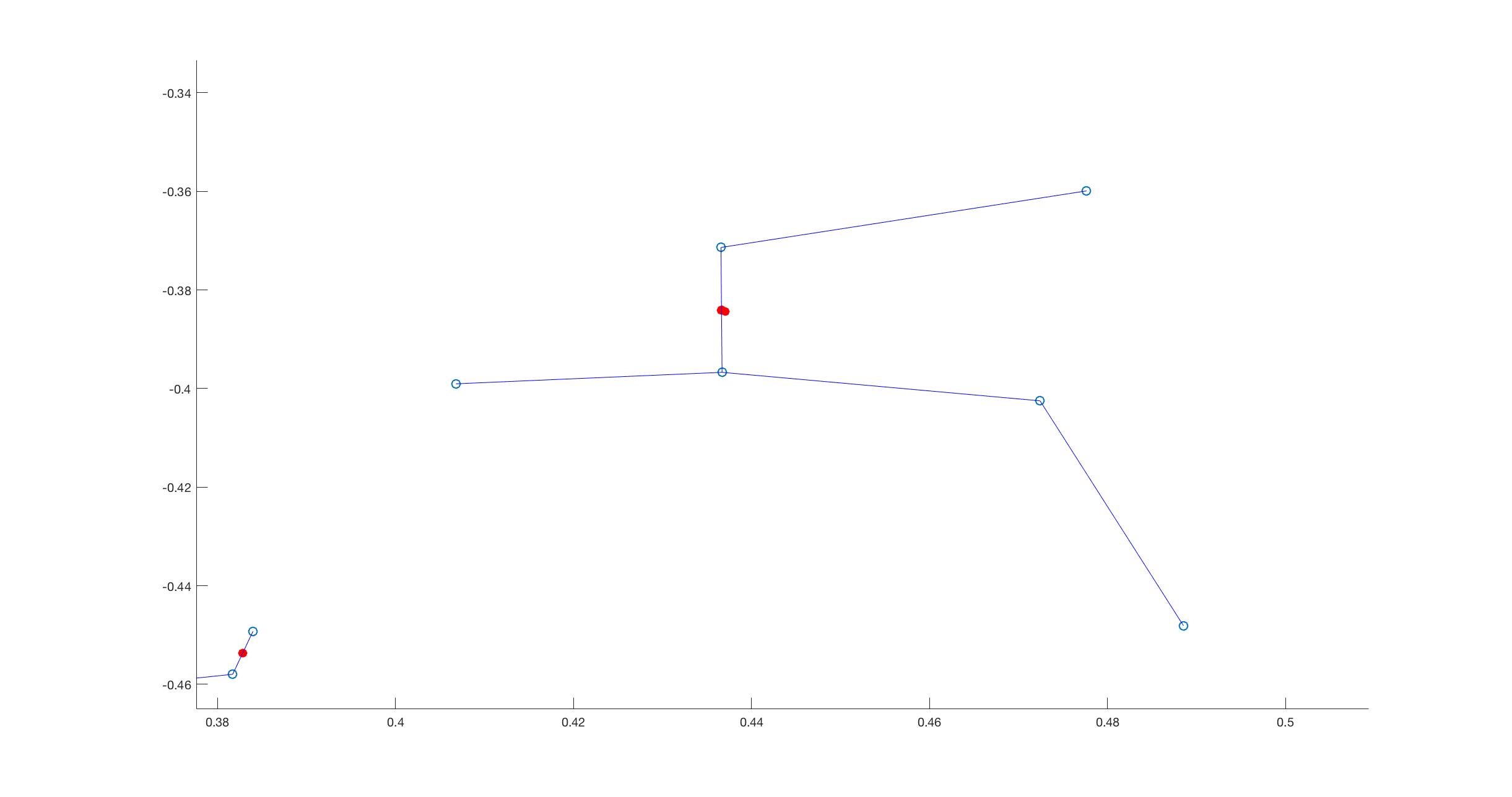}
\end{minipage}%
\begin{minipage}{.5 \textwidth}
  \centering
  \includegraphics[width=.5 \linewidth]{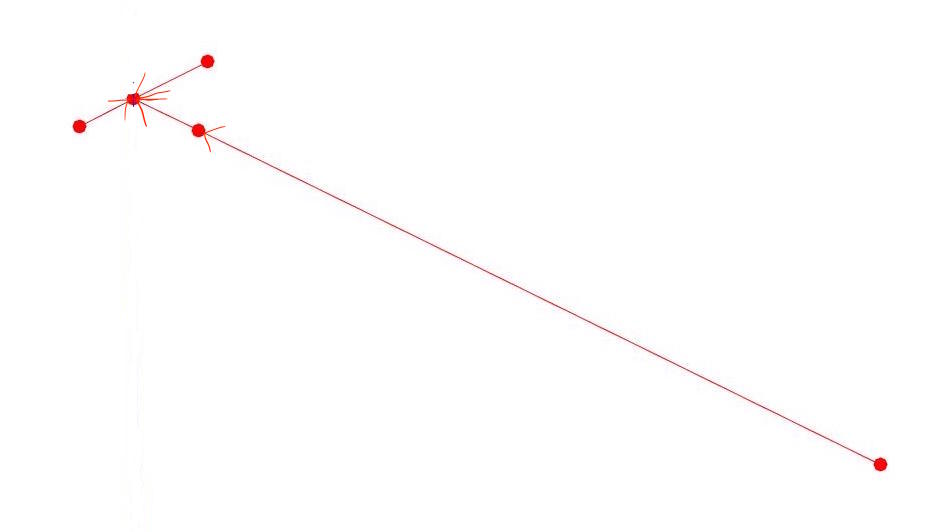}
\end{minipage}
\caption{ \label{fig:whole_chain_7}\textbf{Left:}Tree at the initial time step. \textbf{Right:} Zoomed in situation after 20 time steps}
\end{figure}

After one step of the dynamics, the leader pair becomes fixed. We therefore index all positions relative to this step. Let $a_1=a_1^{(1)}$ denote the (fixed) leader, and let $a_n=a_n^{(1)}$ denote the position of its follower of order $n-1$. The position of agent $n$ at time $i$ is denoted by $a_n^{(i)}$.

The dynamics are governed by the same rule of moving halfway to the leader's position 
\[
a_n^{(i)}=\frac{1}{2}\big(a_n^{(i-1)}+a_{n-1}^{(i-1)}\big),
\]
so each position remains a weighted sum of the initial data.


\begin{lemma}\label{formula_longterm}
If the party is stable, then for every $n \ge 1$ and $i \ge 1$,
\[
 a_n^{(i)}=
    \begin{cases} 
      \frac{1}{2^i}\displaystyle\sum_{k=0}^i {i \choose k} a_{n-k},& i< n,\\[1em]
      \frac{1}{2^i}\left(\displaystyle\sum_{k=0}^{n-1}{i \choose k} a_{n-k} +\sum_{k=n}^{i} {i \choose k} a_1\right), & i \geq n.
   \end{cases}
\]
\end{lemma}

\begin{proof}
The claim follows by induction on $i$, using the linearity of the update rule.
\end{proof}


In particular, for fixed $n$,
\[
\frac{1}{2^i}\sum_{k=0}^{n-1} {i \choose k} a_{n-k} \to 0,
\qquad
\frac{1}{2^i}\sum_{k=n}^{i} {i \choose k} a_1 \to a_1,
\quad \text{as } i \to \infty,
\]
so every agent converges to the leader.


\begin{proposition}\label{prop:no_branching}
If the party is stable, then it cannot have branching away from the root.
\end{proposition}

\begin{proof}
Suppose that branching occurs at some node $a_n$ with $n \ge 2$, and let $a_{n+1,1}$ and $a_{n+1,2}$ denote its two children.

By linearity, common ancestry terms cancel, and the distance between siblings evolves as
\[
d\big(a_{n+1,1}^{(i)}, a_{n+1,2}^{(i)}\big)
= \frac{1}{2^i} d\big(a_{n+1,1}, a_{n+1,2}\big).
\]

On the other hand, using Lemma~\ref{formula_longterm},
\[
a_{n+1,1}^{(i)} - a_n^{(i)}
= \frac{1}{2^i}\sum_{k=0}^{n-1} \binom{i}{k}\big(a_{n+1-k,1} - a_{n-k}\big).
\]
The leading term corresponds to $k=n-1$ and grows like $\binom{i}{n-1}$, while all remaining terms are of strictly lower order. Hence
\[
\big|a_{n+1,1}^{(i)} - a_n^{(i)}\big|
= \frac{1}{2^i}\big(c\, i^{n-1} + O(i^{n-2})\big),
\]
for some $c>0$.

Comparing the two distances, both decay at rate $2^{-i}$, but the distance to the parent carries an additional polynomial factor $i^{n-1}$. Therefore, for all sufficiently large $i$,
\[
d\big(a_{n+1,1}^{(i)}, a_n^{(i)}\big)
>
d\big(a_{n+1,1}^{(i)}, a_{n+1,2}^{(i)}\big),
\]
contradicting stability.
\end{proof}


\noindent
\emph{Remark.} Linearity implies that sibling distances contract purely exponentially, while distances to ancestors retain polynomial memory of the initial configuration, forcing eventual separation.


The proposition suggests that stable configurations are necessarily chain-like. This leads to the following conjectures.

\begin{conjecture}
For every leader pair of any order, there exists a time after which the associated party becomes stable.
\end{conjecture}

\begin{conjecture}
Starting from a Poisson point process in the plane, the limiting configuration consists only of ultimate leader pairs and their followers arranged in chains.
\end{conjecture}

\subsection{Long Term Relations}\label{long_term_relations}

In this subsection we look at some mass transport relations that hold in general in the limiting process. 
Limiting objects consist of ultimate leader pairs and ultimate followers, we can derive a relation between them. Let $\lambda=1$ be the intensity of the initial Poisson point process. Let $\lambda_f$ be the intensity of the ultimate followers and $\lambda_l$ be the intensity of the ultimate leaders. Take $\bar{N}$ to be the mean number of ultimate followers of one ultimate leader. Then by the mass transport principle it directly follows that
\[1\cdot \lambda_f=\bar{N}\lambda_l,\]
using that $\lambda_f+\lambda_l=1$, implies that
\[\bar{N}=\frac{1}{\lambda_l}-1.\]
Since there are two ultimate leaders in a pair, the total number of ultimate followers is $N=2\bar{N}$. 
This relation should hold regardless of the dimensions.

Now if we plug in $\lambda_l=\frac{2}{3}$, like obtained in the simulations, we get that $N=1$. Implying that the mean size of a party is 3 in the limit, which is also consistent with the simulations.

\begin{conjecture}
There exist ultimate followers.
\end{conjecture}

\section{Conclusion}

In summary, we introduced a new model inspired by problems in opinion dynamics. We described various phenomena related to this dynamics and examined the system's long-term behavior. We provided a systematic way to calculate the frequencies of specific configurations, such as the densities of certain leader pairs. Additionally, we numerically analyzed the long-term behavior of parties and derived general results about the shapes of limiting parties. For future work, we plan to prove that parties remain finite at all steps, supported by simulations. Furthermore, we aim to demonstrate that the limiting objects can only be either ultimate leader pairs or ultimate followers. We also plan to explain the observed ratio between the densities of the ultimate leader pair point process and the ultimate follower point process. Finally, we plan to extend these results to higher dimensions.
\begingroup
  
\endgroup


\begin{thebibliography}{9}
  


\bibitem{Acemoglu_opinion}
{\sc D.~Acemoglu and A.~Ozdaglar}, {\em Opinion dynamics and learning in social
  networks}, Dyn. Games Appl., 1 (2011), p.~3–49.

\bibitem{Anderson_2019}
{\sc B.~D.~O. Anderson and M.~Ye}, {\em Recent advances in the modelling and
  analysis of opinion dynamics on influence networks}, Int. J. Autom. Comput.,
  16 (2019), pp.~129--149.

\bibitem{Aydogdu_2017}
{\sc A.~Aydo\u{g}du, M.~Caponigro, S.~McQuade, B.~Piccoli, N.~P. Duteil,
  F.~Rossi, and E.~Tr\'elat}, {\em Interaction network, state space, and
  control in social dynamics}, Modeling and Simulation in Science, Engineering,
  and Technology, Birkh\"auser, 2017.
  
\bibitem{bordenave}
{\sc F.~ Baccelli, C. Bordenave.}{\em The Radial Spanning Tree of a Poisson Point Process}.
The Annals of Applied Probability, vol. 17, no 1, (2007).




\bibitem{baccelli_blas}
{\sc F. ~Baccelli, B. ~Blaszczyszyn. }{\em Stochastic Geometry and Wireless Networks, Volume I -
Theory. }, Foundations and Trends in Networking, pp.150, (2009).

\bibitem{baccelli_blas2}
\leavevmode\vrule height 2pt depth -1.6pt width 23pt, {\em Stochastic Geometry and Wireless Networks, Volume
II - Applications.} pp.209, Foundations and Trends in Networking: Vol. 4: No 1-2, pp 1-312, 978,(2009).

\bibitem{baccelli_nova_knjiga}
{\sc F. ~Baccelli, B. ~Blaszczyszyn, M. K. ~Karray.}
{\em Random Measures, Point Processes, and Stochastic Geometry}
(2020).

\bibitem{chains}
{\sc F. Baccelli, M. Haji-Mirsadeghi.}{\em Point Shift Foliation of a Point Process}.
Electron. J. Probab. Volume 23 (2018), paper no. 19, 25 pp.

\bibitem{berthomieu}
{\sc J. Berthomieu, E. Gillot, M. Safey El Din. }{\em Computing points in connected components defined by a real inequation: algorithms, complexity and implementations, Part I}. preprint, (2026).


\bibitem{other_model1}
{\sc E. Ben-Naim.}{\em Rise and fall of political parties}, Europhys. Lett., vol. 69, no. 5, pp. 671-676, (2005).

\bibitem{Bisin_2000}
{\sc A.~Bisin and T.~Verdier}, {\em Beyond the melting pot: {C}ultural
  transmission, marriage, and the evolution of ethnic and religious traits}, Q.
  J. Econ., 115 (2000), pp.~955--988.

\bibitem{Bisin_2001}
\leavevmode\vrule height 2pt depth -1.6pt width 23pt, {\em The economics of
  cultural transmission and the dynamics of preferences}, J. Econ. Theory, 97
  (2001), pp.~298--319.
  
 

\bibitem{moment_clump}
{\sc B. Blaszczyszyn, C. Rau,V. Schmidt.}{\em Bounds for clump size characteristics in the boolean model}.
Adv. Appl. Prob.(SGSA) 31, 910-928 (1999).

\bibitem{Blondel_2009}
{\sc V.~D. Blondel, J.~M. Hendrickx, and J.~N. Tsitsiklis}, {\em On
  {K}rause’s multi-agent consensus model with state-dependent connectivity},
  IEEE Transactions on Automatic Control, 54 (2009), pp.~2586--2597.
 

\bibitem{Borgers_et_al_candidate_dynamics}
{\sc C.~B\"orgers, B.~Boghosian, N.~Dragovic, and A.~Haensch}, {\em A blue sky bifurcation in the dynamics of political candidates}, American Mathematical Monthly, 131(3), 225–238, (2023). 



\bibitem{Borgers_candidate_voter}
{\sc C.~B\"orgers,  N.~Dragovic,  and A.~Kirshtein}, {\em Candidate Voter Dynamics}, 
Phys A: Stat. Mech. Appl., 131176 (2025)

\bibitem{Boyd_1985}
{\sc R.~Boyd and P.~Richerdson}, {\em Culture and the evolutionary process},
  The University of Chicago Press, Chicago,  (1985).
  
  
\bibitem{kissing_number}
{\sc R.~ Brass, W.O.J.  Moser,  J. Pach.}
{\em Research problems in discrete geometry.} 
Springer. p. 93, (2005).

\bibitem{Canuto_2012}
{\sc C.~Canuto, F.~Fagnani, and P.~Tilli}, {\em An {E}ulerian approach to the
  analysis of {K}rause’s consensus models}, SIAM Journal on Control and
  Optimization, 50 (2012), pp.~243--265.

\bibitem{Cavalli_1981}
{\sc L.~Cavalli-Sforza and M.~Feldman}, {\em Cultural transmission and
  evolution: a quantitative approach}, Princeton University Press, Princeton,
  (1981).

\bibitem{Fewell2006}
{\sc M. Fewell, M. Mod}, {\em Area of common overlap of three circles}, Preprint (2006).

  


\bibitem{Goddard_2022}
{\sc B.~D. Goddard, B.~Gooding, H.~Short, and G.~A. Pavliotis}, {\em Noisy
  bounded confidence models for opinion dynamics: the effect of boundary
  conditions on phase transitions}, IMA Journal of Applied Mathematics, 87
  (2022), pp.~80--110.



\bibitem{haensch2023}
 {\sc  A.~Haensch, N.~Dragovic, C.~B\"orgers, and B.~Boghosian}, {\em A Geospatial Bounded Confidence Model Including Mega-Influencers with an Application to Covid-19 Vaccine Hesitancy}, Journal of Artificial Societies and Social Simulation, 26(1) (2023).
   

\bibitem{nearest_neigh}
{\sc O. Haggstrom, R. Meester }{\em Nearest neighbor and hard sphere models in continuum percolation}.
Rand. Struct. Alg 9, 295-315, (1996).



\bibitem{hegsel_krause_2002}
{\sc R. Hegselmann, U. Krause.}
{\em Opinion dynamics and bounded
confidence models, analysis, and simulation,}
Journal of Artificial Societies and Social Simulation, vol. 5, (2002).

\bibitem{Hegselmann_Krause_2005}
\leavevmode\vrule height 2pt depth -1.6pt width 23pt, {\em Opinion dynamics
  driven by various ways of averaging}, Comput. Econ., 25 (2005), pp.~381--405.
  
 
  
\bibitem{kallenberg}
{\sc O. Kallenberg.}{\em Random measures - Theory and applications}.
(1984).
  
\bibitem{konig_lemma}
{\sc Kleene, S. C. }{\em Mathematical Logic.} New York: Dover, (2002).


\bibitem{krause_1997}
{\sc U.~Krause}, {\em Soziale {D}ynamiken mit vielen {I}nterakteuren. {Eine
  Problemskizze}}, {Modellierung und Simulation von Dynamiken mit vielen
  interagierenden Akteuren}, 3751.2 (1997).

\bibitem{krause_2000}
\leavevmode\vrule height 2pt depth -1.6pt width 23pt, {\em A discrete nonlinear
  and non-autonomous model of consensus formation, {\em in ``Communications in
  Difference Equations"}}, CRC Press, 2000, pp.~227--236.

\bibitem{mohab}
{\sc P. Lairez, M. Mezzarobba, and M. Safey El Din,}{\em Computing the volume of compact semi-algebraic sets},  arXiv preprint, (2019).


\bibitem{liggett_particle_systems}
{\sc T. M. Liggett. } {\em Stochastic Interacting Systems: Contact, Voter and Exclusion Processes}, Springer, (1999).

\bibitem{haenggi} 
{\sc C.  Lee, M. Haenggi }{\em Interference and Outage in Doubly Poisson Cognitive Networks}.  
IEEE Transactions on Wireless Communications 11 (4), 1392-1401, (2012). 

\bibitem{Lord_1979}
{\sc C.~G. Lord, L.~Ross, and M.~R. Lepper}, {\em Biased assimilation and
  attitude polarization: the effects of prior theories on subsequently
  considered evidence}, J. Pers. Soc. Psychol., 37 (1979), pp.~2098--2109.

\bibitem{Lorenz_2005}
{\sc J.~Lorenz}, {\em A stabilization theorem for dynamics of continuous
  opinions}, Physica A,  (2005), pp.~217--223.

\bibitem{Lorenz_2006}
\leavevmode\vrule height 2pt depth -1.6pt width 23pt, {\em Consensus strikes
  back in the {Hegselmann-Krause} model of continuous opinion dynamics under
  bounded confidence}, JASSS, https://www.jasss.org/9/1/8.html,  (2006).

\bibitem{Lorenz_2007}
\leavevmode\vrule height 2pt depth -1.6pt width 23pt, {\em Continuous opinion
  dynamics under bounded confidence: A survey}, Int. J. Modern Phys. C, vol.
  18, no. 12, pp. 1819–1838,  (2007).
  
  
\bibitem{Mastroeni_2019}
{\sc L. ~Mastroeni,  P. ~Vellucci,  and M. ~Naldi.}{\em Agent-Based Models for Opinion Formation: A Bibliographic Survey}, IEEE Access, (2019).




\bibitem{Mirtabatabaei_2012}
{\sc A.~Mirtabatabaei and F.~Bullo}, {\em Opinion dynamics in heterogeneous
  networks: Convergence conjectures and theorems.}, SIAM Journal on Control and
  Optimization, 50(5):2763–2785,  (2012).



\bibitem{Mossel_2017}
{\sc E.~Mossel and O.~Tamuz}, {\em Opinion exchange dynamics}, Probability
  Surveys, 14 (2017).

\bibitem{nedic_high_dim} 
{\sc A. ~Nedic, B. ~Touri.}{\em Multi-Dimensional Hegselmann-Krause Dynamics}.
Proceedings of the 51st IEEE Conference on Decision and Control (CDC), Maui, Hawaii, December 9-13, (2012), pp. 68-73.


\bibitem{amos_nevo}
{\sc A. Nevo,}{\em Pointwise Ergodic Theorems for Actions of Groups}.
Handbook of Dynamical Systems, vol. 1B.  (2005).

\bibitem{Piccoli_2021}
{\sc B.~Piccoli and F. Rossi}, {\em Generalized solutions to
  bounded-confidence models}, Mathematical Models and Methods in Applied
  Sciences Vol. 31, No. 06, pp. 1237-1276,  (2021).



\bibitem{Proskurnikov_2017}
{\sc A.~V. Proskurnikov and R. ~Tempo}, {\em A tutorial on modeling and
  analysis of dynamic social networks. part {I}}, Annual Reviews in Control,
  43:65 – 79,  (2017).

\bibitem{Proskurnikov_2018}
\leavevmode\vrule height 2pt depth -1.6pt width 23pt, {\em A tutorial on
  modeling and analysis of dynamic social networks. part {II}}, Annual Reviews
  in Control, 45:166 – 190,  (2018)
  





%
%







%

\bibitem{Perrier_2024}
{\sc  R.~Perrier, H.~Schawe, and L.~Hernández},{\em Phase coexistence in the fully heterogeneous Hegselmann–Krause opinion dynamics model}, Sci Rep 14, 241 (2024). 







\bibitem{ramesh}
{\sc L. Ramesh and N. Weiss,}{\em Exact Volumes of Semi-Algebraic Convex Bodies},ArXiv Preprint (2026).

\bibitem{knjiga}
{\sc G. Simmons. }
{\em Introduction to Topology and Modern Analysis.}
McGraw-Hill, New York, (1963).

\bibitem{kendall_mecke}
{\sc D. Stoyan, W. S. Kendall, J. Mecke. }
{\em Stochastic geometry and its applications.}
Wiley Series in Probability and Mathematical Statistics: Applied Probability and Statistics. John Wiley \& Sons Ltd., Chichester, (1987).

\bibitem{weil_geometry}
{\sc R. ~Schneider and W. ~Weil. }{\em Stochastic and Integral Geometry}. Springer,
(2008). 
 

















\bibitem{Wedin_2015}
{\sc E.~Wedin and P.~Hegarty}, {\em The {H}egselmann-{K}rause dynamics for the
  continuous-agent model and a regular opinion function do not always lead to
  consensus}, IEEE Transactions on Automatic Control, 60 (2015),
  pp.~2416--2421.




\end{thebibliography}
\end{document}